\documentclass[11pt,a4paper,oneside]{amsart}

\usepackage[utf8]{inputenc}

\usepackage{fullpage}
\usepackage{amsmath, amsthm, amssymb, mathtools, mathrsfs}
\usepackage{enumerate}
\usepackage[all]{xy}
\usepackage{array, booktabs}
\usepackage{float}
\usepackage{mathtools}
\usepackage{extarrows}
\usepackage{tabularx}
\usepackage{graphicx,color}
\usepackage{colortbl}%表の色付け

\newtheorem{theoremcounter}{Theorem Counter}[section]

\theoremstyle{plain} %Theorem 環境の見出しを太字、本文に斜体を使う
\newtheorem{definition}[theoremcounter]{Definition}
\newtheorem{theorem}[theoremcounter]{Theorem}
\newtheorem{lemma}[theoremcounter]{Lemma}
\newtheorem{proposition}[theoremcounter]{Proposition}
\newtheorem{corollary}[theoremcounter]{Corollary}
\newtheorem{conjecture}[theoremcounter]{Conjecture}

\theoremstyle{definition} %Theorem 環境の見出しを太字、本文に斜体を使わない
\newtheorem{remark}[theoremcounter]{Remark}

\newtheorem*{notation*}{Notation}
\numberwithin{equation}{section}

\newcommand{\iu}{\sqrt{-1}}%imaginary unitの頭文字

\newcommand{\Z}{\mathbb{Z}}
\newcommand{\Q}{\mathbb{Q}}
\newcommand{\R}{\mathbb{R}}
\newcommand{\C}{\mathbb{C}}
\newcommand{\bbH}{\mathbb{H}}
\newcommand{\bbP}{\mathbb{P}}

\newcommand{\calQ}{\mathcal{Q}}

\newcommand{\veps}{\varepsilon}

\DeclareMathOperator{\ReNew}{Re}

\DeclareMathOperator{\SL}{SL}

\DeclareMathOperator{\GL}{GL}

\newcommand{\disc}{\mathop{\mathrm{disc}}\nolimits}

\newcommand{\relmiddle}[1]{\mathrel{}\middle#1\mathrel{}}%括弧に合わせて大きくなる\mid. \relmiddle|のように使う

\newcommand{\pmat}[1]{\begin{pmatrix}#1\end{pmatrix}}
\newcommand{\smat}[1]{\bigl(\begin{smallmatrix}#1\end{smallmatrix}\bigr)}

\newcommand{\abs}[1]{\left\lvert#1\right\rvert}

\newcommand{\thmref}[2]{\hyperref[#2]{#1 \ref*{#2}}}
\renewcommand{\eqref}[1]{\hyperref[#1]{(\ref*{#1})}}

\makeatletter
\renewcommand{\theenumi}{{\upshape (\@roman\c@enumi)}}

\renewcommand{\p@enumii}{}
\renewcommand{\theenumii}{{\upshape (\@alph\c@enumii)}}

\makeatother

\renewcommand{\emph}[1]{\textbf{#1}}

\usepackage[dvipdfmx]{hyperref}
\hypersetup{% hyperrefオプションリスト
colorlinks=true,
linkcolor=blue,
citecolor=red,
}

\begin{document}
% --------------------------------------------------------------------------

\title{Hurwitz class numbers with level and modular correspondences}
\author{Yuya Murakami} 
\thanks{Mathematical Inst. Tohoku Univ., 6-3, Aoba, Aramaki, Aoba-ku, Sendai 980-8578, JAPAN.
\textit{E-mail address}: \texttt{yuya.murakami.s8@dc.tohoku.ac.jp}} 
\date{\today}
\maketitle
%\tableofcontents

% --------------------------------------------------------------------------

\begin{abstract}
	In this paper, we prove Hurwitz-Eichler type formulas for Hurwitz class numbers with each level $ M $ when the modular curve $ X_0(M) $ has genus zero.
	A key idea is to calculate intersection numbers of modular correspondences with the level in two different ways.
	A generalization of Atkin-Lehner involutions for $ \Gamma_0(M) $ and its subgroup $ \Gamma_0^{(M')}(M) $ is introduced to calculate intersection multiplicities of modular correspondences at cusps.
\end{abstract}

% --------------------------------------------------------------------------

\section{Introduction and statement of results} \label{sec:intro}

% --------------------------------------------------------------------------

For a positive integer $ M, D $ with $ D \equiv 0, 3 \bmod 4 $, let us define
\[
H^M(D) := \sum_{[Q] \in \calQ_{-D, >0}^{M}/\Gamma_0(M)}
\frac{2}{\# \Gamma_0(M)_Q}
\]
and call it the $ D $th Hurwitz class number of level $ M $.
Here, for integers $ a, b, c $, let us write a quadratic form $ [a, b, c] := a X^2 + bXY + c Y^2 $ whose discriminant is $ \disc Q := b^2 - 4ac $ and let
\[
\calQ_{-D, >0}^{M} :=
\left\{ Q= [Ma, b, c] \mid a, b, c \in \Z, a>0, \disc Q = -D \right\}.
\]
%Let $ \SL_2(\Z) := \left\{ \gamma \in M_2(\Z) \relmiddle| \det \gamma = 1 \right\} $.
The group 
\[
\Gamma_0(M) := \left\{ \gamma \in \SL_2(\Z) \relmiddle| \gamma \equiv \begin{pmatrix} * & * \\ 0 & * \end{pmatrix} \bmod M \right\}
\]
acts on $ \calQ_{-D, >0}^{M} $ by
\[
\left( Q \circ \begin{pmatrix} a & b \\ c & d \end{pmatrix} \right) (X, Y)
:= Q( aX+bY, cX+dY ), \quad  
Q \in \calQ_{-D, >0}^{M}, \quad  
\begin{pmatrix} a & b \\ c & d \end{pmatrix} \in \Gamma_0(M).
\]
We denote by $ \Gamma_0(M)_Q $ the stabilizer of $ Q \in \calQ_{-D, >0}^{M} $ under this action.
To compute $ \calQ_{-D, >0}^{M}/\Gamma_0(M) $ is equivalent to understand imaginary quadratic points with discriminant $ -D $ on a suitable fundamental domain for the modular curve 
$Y_0(M):=\Gamma_0(M)\backslash \mathbb{H}$ and the related reduction theory as well. 
This will be carried out in Section \ref{sec:computation}. 

When $ M=1 $, we put $ H(D) := H^1(D) $.
For a positive integer $ N $ which is not a square, the following relation is known as Hurwitz-Eichler relation:
\begin{equation} \label{eq:Hurwitz-Eichler}
	\sum_{x \in \Z,\ x^2 < 4N} H(4N- x^2)
	= \sum_{ad=N} \max \{ a, d \}.
\end{equation}
%\cite[Theorem 5.3.8]{Cohen}}
%A detailed proof will be found in \cite{GK} or \cite{Vog} with the result in \cite{Gor}.
Eichler's original proof is found in \cite{Eichler}.
Another proof will be found in \cite{Ling} by calculating intersection multiplicities at cusps and intersection number of certain algebraic cycles on $ \bbP^1 \times \bbP^1 $ which are called modular correspondences.

In this paper, we consider an analog of the relation (\ref{eq:Hurwitz-Eichler}) for $ M > 1 $ such that the genus of the modular curve $ X_0(M) $ is zero.
Our main result is the following theorem:

\begin{theorem}
	\label{thm:main_class_num}
	Let $ M $ be $ 2 \le M \le 10$ or $ M \in \{ 12, 13, 16, 18, 25 \} $.
	Let $ N $ be a positive integer which is coprime to $ M $ and is not a square.
	It holds that
	\begin{enumerate}
		\item 
			\[
			\sum_{x \in \Z,\ x^2 < 4N} 
			H^M \left( 4N - x^2 \right)
			= \sum_{a d = N} |a - d|
			\]
			if either of the following conditions is satisfied:
			\begin{enumerate}
				\item $ M \in \{ 2,3,5,7,13 \}, $
				\item $ M = 9 $ and $ N \equiv -1 \bmod 3 $,
				\item $ M = 25 $ and $ N \equiv \pm 2 \bmod 5 $,
			\end{enumerate}
		\item 
			\[
			\sum_{x \in \Z,\ x^2 < 4N} 
			H^{4} \left( 4N - x^2 \right)
			= 2 \sum_{a d = N, a>d} (a-2d)
			\]
			if $ M = 4 $,
		\item 
			\[
			\sum_{x \in \Z,\ x^2 < 4N} 
			H^M \left( 4N - x^2 \right)
			= 2 \sum_{a d = N, a>d} (a-3d)
			\]
			if either of the following conditions is satisfied:
			\begin{enumerate}
				\item $ M \in \{ 6, 8, 10 \}, $
				\item $ M = 9 $ and $ N \equiv 1 \bmod 3 $,
				\item $ M = 16 $ and $ N \equiv -1 \bmod 4 $,
				\item $ M = 18 $ and $ N \equiv -1 \bmod 6 $,
			\end{enumerate}
		\item 
			\[
			\sum_{x \in \Z,\ x^2 < 4N} 
			H^{12} \left( 4N - x^2 \right)
			= 2 \sum_{a d = N, a>d} (a-5d)
			\]
			if either of the following conditions is satisfied:
			\begin{enumerate}
				\item $ M = 12 $,
				\item $ M = 16 $ and $ N \equiv 1 \bmod 4 $,
			\end{enumerate}
		\item 
			\[
			\sum_{x \in \Z,\ x^2 < 4N} 
			H^{18} \left( 4N - x^2 \right)
			= 2 \sum_{a d = N, a>d} (a-7d)
			\]
			if $ M = 18 $ and $ N \equiv 1 \bmod 6 $,
		\item
			 \[
			 \sum_{x \in \Z,\ x^2 < 4N} 
			 H^{25} \left( 4N - x^2 \right)
			 = \sum_{a d = N} |a - d| 
			 - 8\sum_{ad = N, a>d, a \equiv d \bmod 5}
			 d
			 \]
			 if $ M = 25 $ and $ N \equiv \pm 1 \bmod 5 $.		
	\end{enumerate}
\end{theorem}

To prove Theorem \ref{thm:main_class_num}, we calculate both sides in two ways as is done in the proof of \cite[Corollary 1.1]{Ling}.
In our calculation, we use the modular correspondence $T_{N}^{\Gamma_0(M)}$ with level $ M $ and degree $ N $, which is our main theme.

We state a definition of modular correspondences.
Let 
\[
\bbH := \{ \tau = x+y\sqrt{-1} \mid x, y \in \R, y>0 \}
\]
be the complex upper half-plane. 
For a positive integer $ M $, we define the modular curves of level $ M $ as
\[
Y_0(M) := \Gamma_0(M) \backslash \bbH, \quad  
X_0(M) := \Gamma_0(M) \backslash (\bbH \cup \Q \cup \{ i\infty \}).
\]
They admit the structure as Riemann surfaces and it turns out that $ X_0(M) $ is compact. 
%A pair $(E, C)$ is an elliptic curve with level structure for the congruence subgroup $\Gamma_0(M)$ if $ E $ is an elliptic curve over $ \C $ and $ C $ is a cyclic subgroup of order $ M $ in $ E $. 
%The modular curve $ Y_0(M) $ identifies naturally the set of isometric classes of elliptic curves with level structure for $\Gamma_0(M)$. 
Each element in $ \Gamma_0(M) \backslash (\Q \cup \{ i\infty \}) $ is called a cusp.

In this paper, we assume that the modular curve $ X_0(M) $ has genus zero.
It is well-known that such $ M $ is $ 1 \le M \le 10 $ or $ M= 12, 13, 16, 18, 25 $ (\cite[Section 3]{Seb}).
For a positive integer $ N $ coprime to $ M $, the modular correspondence of degree $ N $ with respect to $ \Gamma_0(M) $ introduced in \cite{Mura} is defined by
\begin{equation} \label{eq:def_of_T}
	T_{N}^{\Gamma_0(M)} := 
	\bigcup_{A = \smat{a & b \\ 0 & d} \in \mathrm{M}_2({\Z}), \,
		ad=N, \, 0 \le b < d}
	\left\{
	(\Gamma_0(M) \tau, \Gamma_0(M) A(\tau)) \in X_0(M) \times X_0(M)
	\right\}.
\end{equation}
It turns out that the modular correspondence $ T_{N}^{\Gamma_0(M)} $ is an algebraic cycle in $ X_0(M) \times X_0(M) $ by \cite[Theorem 2.9]{Mura}.
The modular correspondence $ T_{N}^{\Gamma_0(M)} $ and the diagonal set 
\[
\Delta := \{ (\Gamma_0(M) \tau, \Gamma_0(M) \tau) \in X_0(M) \times X_0(M) \}
\]
intersect properly when $ N $ is not a square and the intersection number of them on $ Y_0(M) \times Y_0(M) $ coincides with the left-hand side in Theorem \ref{thm:main_class_num} by \cite[Theorem 1.2]{Mura}.

On the other hand, we can calculate the intersection number on $  Y_0(M) \times Y_0(M) $ by subtracting it on $ X_0(M) \times X_0(M) \smallsetminus Y_0(M) \times Y_0(M) $ from it on $ X_0(M) \times X_0(M) $.
The result coincides with the right-hand side in Theorem \ref{thm:main_class_num} by the following theorem. 

%Here we state this result for a part of cases and state for all cases in Section \ref{sec:mult_cusp} and \ref{sec:class_num_formula}. 

\begin{theorem} \label{thm:main_cusp_intro}
	Let $ M \in \{ 2, 3, 5, 6, 7, 8, 10, 12, 13 \} $ and $ N $ be a positive integer coprime to $ M $ which is not a square. 
	Then the following holds. 
	\begin{enumerate}
		\item The intersection number of $ \Delta $ and $ T_{N}^{\Gamma_0(M)}$ is
		\[
		(\Delta \cdot T_{N}^{\Gamma_0(M)})_{X_0(M) \times X_0(M)}
		= 2 \sum_{d \mid N} d.
		\]
		\item \label{item:thm:main_cusp_intro}
		The intersection multiplicity of $ \Delta $ and $ T_{N}^{\Gamma_0(M)}$ at a pair $ (s, s) $ of cusp $ s $ is
		\[
		(\Delta \cdot T_{N}^{\Gamma_0(M)})_{(s, s)}
		= 2 \sum_{a d = N, a > d} d.
		\]
		\item We have
		\[
		(\Delta \cdot T_{N}^{\Gamma_0(M)})_{Y_0(M) \times Y_0(M)}
		= 2 \sum_{a d = N, a > d} (a - (c_0(M) - 1) d)
		\]
		where
		\[
		c_0(M) :=
		\# \{ \text{cusps in } X_0(M) \}
		=
		\begin{cases}
		2 & \text{ if } M \in \{ 2, 3, 5, 7, 13 \}, \\
		3 & \text{ if } M = 4, \\
		4 & \text{ if } M \in \{ 6, 8, 10 \}, \\
		6 & \text{ if } M = 12. \\
		\end{cases}
		\]
	\end{enumerate} 
\end{theorem}

In the above theorem, we treat only the case when $ M \in \{ 2, 3, 5, 6, 7, 8, 10, 12, 13 \} $ for simplicity.
In other remaining cases, we state similar results in Section \ref{sec:class_num_formula}.

%In the case $ M =  2, 3, 5, 6, 7, 8, 10, 12, 13 $, Theorem \ref{thm:main_class_num} follows from Theorem \ref{thm:main_Y_0(M)} and Theorem \ref{thm:main_cusp_intro}.

The most important part of Theorem \ref{thm:main_cusp_intro} is to calculate the intersection multiplicities at cusps in \ref{item:thm:main_cusp_intro}.
Although we can achieve it by using Atkin-Lehner involutions, they exist less than cusps for some levels and thus we introduce generalized Atkin-Lehner involutions.

Recently, Brunier-Schwagenscheidt gave various interesting formulas involving our generalized Hurwitz class numbers in a different context \cite[Example 4.2]{BruSch}. 
As the second author of loc.cit. commented to the author, it would be interesting to study any relation between our work and theirs.  

This paper will be organized as follows. 
In Section \ref{sec:known}, we summarize known results for our modular correspondences $ T_{N}^{\Gamma_0(M)} $.
In Section \ref{sec:G_0(M)}, we introduce a subgroup $ G_0(M) \subset \GL_2(\Q) $ which contains the above matrix $ A $.
In Section \ref{sec:involution}, for each cusp $ s $, we give explicitly a matrix $ W \in \GL_2(\Q) \cap M_2(\Z) $ which satisfies $ W(i\infty) = s $ and normalize $ \Gamma_0(M) $.
When $ M $ is square-free, each cusp is represented by the image of $ i\infty $ under an Atkin-Lehner involution.
However, this is not true when $ M $ is not square-free.
then there exists a cusp such that there does not exist an 
For this reason, we introduce a generalization of Atkin-Lehner involutions.
In Section \ref{sec:normalizer}, we give a condition whether such matrices introduced in Section \ref{sec:involution} normalize $ \Gamma_0(M) $.
In Section \ref{sec:action_on_cusps}, we classify cusps.
We calculate the intersection multiplicities at cusps in Section \ref{sec:mult_cusp} and prove Theorem \ref{thm:main_class_num} in Section \ref{sec:class_num_formula}.
In Section \ref{sec:computation}, we give explicit computation of the Hurwitz class number $ H^M(D) $ for $ M $ when $ 2 \le M \le 10 $ or $ M = 13 $.
In Section \ref{sec:exapmles}, we give some examples of Theorem \ref{thm:main_class_num} for small $ N $ and conjecture for a square $ N $.

%In Section \ref{sec:G_0(M)}, we define and consider a subgroup $ G_0(M) \subset \GL_2(\Q) $ and give a criterion of the above condition in  (\ref{eq:T_N_point_condition}).
%In Section \ref{sec:involution}, for each cusp $ s $ we enumerate a matrix $ W \in \GL_2(\Q) \cap M_2(\Z) $ which satisfies $ W(i\infty)=s $ and normalize $ \Gamma_0(M) $ and $ G_0(M) $.
%Typical such matrix is an Atkin-Lehner involution.
%However, in the case when $ M $ is not square-free, there exists a cusp such that there does not exist an Atkin-Lehner involution $ W $ with $ W(i\infty)=s $.
%In this reason, we introduce a type of matrices which is a generalization of Atkin-Lehner involutions.
%When $ M = 25 $, however, such matrices normalize neither $ \Gamma_0(M) $ nor $ G_0(M) $ and that is the reason why our formula in Theorem \ref{thm:main_class_num} is complicated in this case.
%In Section \ref{sec:normalizer}, we determine the normalizer of $ \Gamma_0(M) $ and $ G_0(M) $ and give a condition whether a generalization of Atkin-Lehner involutions introduced in Section \ref{sec:involution} normalizes $ \Gamma_0(M) $ or $ G_0(M) $.
%In Section \ref{sec:mult_cusp}, we calculate the intersection multiplicities at cusps.
%Using the results there, we prove Theorem \ref{thm:main_class_num} in Section \ref{sec:class_num_formula}.
%In Section \ref{sec:exapmles}, we give some examples for small $ N $ and conjecture for a square $ N $ of Theorem \ref{thm:main_class_num}.

% --------------------------------------------------------------------------

\section*{Acknowledgement} \label{sec:acknowledgement}

% --------------------------------------------------------------------------

I would like to show my greatest appreciation to Professor Takuya Yamauchi for giving many pieces of advice. 
I am deeply grateful to Dr.~Toshiki Matsusaka for giving many comments and pointing out that the 0th Hurwitz class number has a form $ -(p+1)/12 $ when a level is a prime number $ p $.
I would like to express my gratitude to Dr.~Markus Schwagenscheidt for telling me the relation between my work and theta lift.
The author is supported by JSPS KAKENHI Grant Number JP 20J20308.

% --------------------------------------------------------------------------

\section{Known results} \label{sec:known}

% --------------------------------------------------------------------------

In this section, we summarize results for modular correspondences in \cite{Mura}.
Let $ M $ be a positive integer such that the modular curve $ X_0(M) $ has genus zero.
In this case, there exists the unique isomorphism $ t \colon X_0(M) \xrightarrow{\sim} \bbP^1(\C) $ satisfying
$ \mathrm{div}(t) = (0) - (\infty) $ and having an expansion
\[
t(\tau) = q^{-1} + c_0 + c_1 q + \cdots \in q^{-1} + \Z[[q]] 
\]
with $ q := e^{2 \pi \iu \tau} $ for $ M>1 $. 
Such $ t $ is given as an explicit products of the Dedekind eta function in \cite[Table 1]{Mura} which refers to \cite[Subsection 3.1]{Mai}.
For $ M=1 $, we put $ t \colon X_0(1) \xrightarrow{\sim} \bbP^1(\C) $ as the $ j $-invariant.

The modular correspondence $ T_{N}^{\Gamma_0(M)} $ defined in (\ref{eq:def_of_T}) is an algebraic cycle in $ X_0(M) \times X_0(M) $ by the following theorem.
We remark that the definition of the modular correspondence $ T_{N}^{\Gamma_0(M)} $ in (\ref{eq:def_of_T}) differs from the original definition in \cite{Mura} but it is essentially the same.

\begin{theorem}[{\cite[Theorem 2.9]{Mura}}]\label{thm:main_mod_poly}
	For a positive integer $ N $ coprime to $ M $, the image of the modular correspondence $ T_{N}^{\Gamma_0(M)} $ under the map
	$ t \times t \colon X_0(M) \times X_0(M) \to \bbP^1(\C) \times \bbP^1(\C) $
	is the zero set of a polynomial
	$ \Phi_{N}^{\Gamma_0(M)}(X, Y) \in \Z[X, Y] $.
\end{theorem}

The polynomial $ \Phi_{N}^{\Gamma_0(M)}(X, Y) $ is called the modular polynomial of level $ M $ and degree $ N $.

The following theorem states that the intersection number of two modular correspondences on $ Y_0(M) \times Y_0(M) $ is equal to the left-hand side in Theorem \ref{thm:main_class_num}. 

\begin{theorem}[{\cite[Theorem 1.2]{Mura}}]\label{thm:main_Y_0(M)}
	For positive integers $ N_1 $ and $ N_2 $ coprime to $ M $,
	two algebraic cycles $T_{N_1}^{\Gamma_0(M)} $ and $ T_{N_2}^{\Gamma_0(M)}$ intersect properly if and only if the integer $N_1 N_2$ is not a square. 
	Moreover, in this case, the intersection number on $ Y_0(M) \times Y_0(M) $ is given as
	\begin{align*}
	(T_{N_1}^{\Gamma_0(M)} \cdot T_{N_2}^{\Gamma_0(M)})
	_{Y_0(M) \times Y_0(M)}
	&:= \sum_{(x_0, y_0) \in Y_0(M) \times Y_0(M)} (T_{N_1}^{\Gamma_0(M)} \cdot T_{N_2}^{\Gamma_0(M)})_{(x_0, y_0)} \\
	&=\sum_{x \in \Z,\ x^2 < 4N_1 N_2} \sum_{d \mid (N_1, N_2, x)}
	d \cdot H^M \left( \frac{4N_1 N_2- x^2}{d^2} \right).
	\end{align*}
	In particular, for a non-square positive integer $ N $ coprime to $ M $, we have
	\[
	(T_{1}^{\Gamma_0(M)} \cdot T_{N}^{\Gamma_0(M)})
	_{Y_0(M) \times Y_0(M)}
	= \sum_{x \in \Z,\ x^2 < 4N} 
	H^M \left( 4N- x^2 \right).
	\]
\end{theorem}

Here we remark that $ T_{1}^{\Gamma_0(M)} $ is equal to the diagonal set $ \Delta $ by (\ref{eq:def_of_T}).

% --------------------------------------------------------------------------

\section{A subgroup $ G_0(M) $ in $ \GL_2(\Q) $} \label{sec:G_0(M)}

% --------------------------------------------------------------------------

In this section, we introduce and study the subgroup $ G_0(M) $ of $ \GL_2(\Q) $ which plays an important role in studying intersections of modular correspondences at cusps.

Let $ M $ be a positive integer. 

\begin{definition}
	Set
	\[
	\Z_{(M)} :=
	\left\{ \frac{a}{b} \in \Q \relmiddle{|} a, b \in \Z, (b, M) = 1 \right\},
	\]
	\[
	\GL_2(\Z_{(M)}) :=
	\left\{ \gamma \in M_2(\Z_{(M)}) \mid \det \gamma \in \Z_{(M)}^{\times} \right\},
	\]
	and
	\[
	G_0(M) :=
	\left\{ \gamma \in \GL_2(\Z_{(M)}) \relmiddle{|} 
	\gamma \equiv \begin{pmatrix} * & * \\ 0 & * \end{pmatrix}
	\bmod M \Z_{(M)} \right\}.
	\]
\end{definition}

If $ M $ is a prime number $ p $, then $ \Z_{(p)} $ is a localization of $ \Z $ at the prime ideal $ (p) := p \Z $. 
We also remark that $ \Gamma_0(M) = G_0(M) \cap \SL_2(\Z) $. 

Our aim in this section is to prove the following proposition which plays an important role in studying the action of $ G_0(M) $ on cusps in the next section.

\begin{proposition} \label{prop:str_G_0(M)}
	It holds that $ G_0(M) = \Gamma_0(M) G_0(M)_{i \infty} $.
\end{proposition}

Before giving a proof, we prepare the followings. 

\begin{definition}
	For rational numbers $ a $ and $ b $, we have a unique rational number $ g \in \Q_{\ge 0} $ such that $ a \Z + b \Z = g \Z $.
	We put $ (a, b) := g $ and call it the greatest common divisor of $ a $ and $ b $.
\end{definition}

If $ g \neq 0 $, then $ a/g, b/g \in \Z $.
In the case when both $ a $ and $ b $ are integers, the above $ (a, b) $ is the usual greatest common divisor of $ a $ and $ b $.

The following property of the greatest common divisor is quite elementary.

\begin{lemma} \label{lem:gcd}
	For rational numbers $ a \in \Z_{(M)}^{\times} $ and $ b \in \Z_{(M)} $, we have $ (a, b) \in \Z_{(M)}^{\times} $. 
\end{lemma}

\begin{proof}
	Let $ g := (a, b) \in \Z_{(M)} $.
	Since $ g^{-1} a \in \Z $, we have $ g \in \Z_{(M)}^{\times} $. 
\end{proof}

\begin{proof}[Proof of Proposition $ \ref{prop:str_G_0(M)} $]
	For a matrix
	\[
	A = \begin{pmatrix} a & b \\ Mc & d \end{pmatrix} \in G_0(M),
	\]
	we have $ a \in \Z_{(M)}^{\times} $
	since $ D := ad - Mbc \in \Z_{(M)}^{\times} $. 
	By Lemma \ref{lem:gcd}, we have $ (a, c) \in \Z_{(M)}^{\times} $. 
	Thus there exists a matrix
	\[
	\gamma
	= \begin{pmatrix} a/(a, c) & * \\ Mc/(a, c) & * \end{pmatrix}
	\in \Gamma_0(M).
	\]
	We have $ A(i\infty) = \gamma(i\infty) $.
\end{proof}

% --------------------------------------------------------------------------

\section{Cusps and Atkin-Lehner involutions} \label{sec:involution}

% --------------------------------------------------------------------------

In this section, for each cusp $ s $ in $ X_0(M) $, we consider whether there exists a matrix $ W \in \mathrm{SL}_2(\R) $ which satisfies $ W(i\infty)=s $ and normalize both of $ \Gamma_0(M) $ and $ G_0(M) $.
Since all cusps are expressed as the form $ m/M $ with an integer $ 0 \le m < M $, we need a matrix $ W \in \SL_2(\R) $ with the form
\[
\frac{1}{\sqrt{D}} \begin{pmatrix} m & u \\ M & v \end{pmatrix}.
\]
Typical such matrices are Atkin-Lehner involutions. 

\begin{definition}
	For a positive divisor $ m $ of a positive integer $ M $ such that $ (m, M/m)=1 $, there exist integers $ u, v $ such that $ mv - Mu/m = 1 $. 
	We denote
	\[
	W_m = W_m^M := \frac{1}{\sqrt{m}} \begin{pmatrix} m & u \\ M & mv \end{pmatrix} \in \SL_2(\R).
	\]
	For $ m = 0 $, we set
	\begin{equation} \label{eq:Atkin-Lehner} 
	W_0 = W_0^M := \frac{1}{\sqrt{M}} \begin{pmatrix} 0 & -1 \\ M & 0 \end{pmatrix} \in \SL_2(\R).
	\end{equation}
	We call them \textbf{Atkin-Lehner involutions}. 	
\end{definition}

We can check that Atkin-Lehner involutions normalize $ \Gamma_0(M) $ and $ G_0(M) $ by direct calculation.

If $ M $ is square-free, one can find an Atkin-Lehner involution $ W $ such that $ W(i\infty) = s $ for each cusp $ s \in X_0(M) $.
However, this is not true if $ M $ is not square-free.

For this reason, we introduce the following generalization of Atkin-Lehner involutions.

\begin{definition} \label{def:other_aut}
	For a positive integer $ M $ and an integer $ m $, let $ D := (M, m^2) $. Take $ u $ and $ v $ such that $ (M, m)mv - Mu = D $ and sert
	We set
	\[
	W_m = W_m^M := \frac{1}{\sqrt{D}} 
	\begin{pmatrix} m & u \\ M & (M, m)v \end{pmatrix}
	\in \SL_2(\R).
	\]
	It is called a \textbf{generalized Atkin-Lehner involution}. 	
	%	
	%	By Proposition $ \ref{prop:aut_of_X_0(M)} $, we have
	%	$ W_m \in N_{\SL_2(\R)}(\Gamma_0^{(M')}(M)) $ for $ M' := M/(M, m) $.
\end{definition}

In the case when $ m $ is a positive divisor of $ M $ such that $ (m, M/m)=1 $,
$ W_m $ is an Atkin-Lehner involution. 

In general, generalized Atkin-Lehner involutions do not normalize neither $ \Gamma_0(M) $ nor $ G_0(M) $.
For example, if $ M=25 $ and $ (m, 25) = 5 $, then it turns out that
$ W_m^M $ do not normalize neither $ \Gamma_0(M) $ nor $ G_0(M) $ in Section \ref{sec:normalizer}.
To calculate the intersection multiplicity of the modular correspondences at a pair of cusps in such a case, we define the following subgroup of $ \Gamma_0(M) $.

\begin{definition}
	For a positive integer $ M $ and its positive divisor $ M' $, put
	\[
	\Gamma_0^{(M')}(M) := 
	\left\{ \begin{pmatrix} a & * \\ * & d \end{pmatrix} \in \Gamma_0(M)  \relmiddle| a \equiv d \bmod M' \right\}.
	\]
	This group is a congruence subgroup of level $ M $ since
	$ \Gamma(M) \subset \Gamma_0^{(M')}(M) $. 
	We put $ X_0^{(M')}(M) := \Gamma_0^{(M')}(M) \backslash (\Q \cup \{ i \infty \}) \cup \bbH $ be the associated modular curve. 
\end{definition}

For example, $ \Gamma_0^{(1)}(M) = \Gamma_0(M) $. 

To conclude this section, we enumerate cusps $ m/M $ and generalized Atkin-Lehner involutions $ W_m^M $ when the genus of the modular curve $ X_0(M) $ is zero, that is exactly when $ 1 \le M \le 10 $ or $ M \in \{ 12, 13, 16, 18, 25 \} $.

Here we remark that $ i\infty $ and $ 1/M $ are $ \Gamma_0(M) $-equivalent.
When $ M=1 $, $ X_0(1) $ has only one cusp $ i\infty $ and $ X_0(M) $ has two cusps $ i\infty $ and $ 0 $ for a prime number $ M $. 
For a composite number $ M $, $ X_0(M) $ has one or more cusps except for $ i\infty, 0 $. 

In the case when $ M $ is a composite number, we list the cusps in $ X_0(M) $ in Table \ref{tab:cusps} and Atkin-Lehner involutions $ W_m^M $ for $ m \neq 1, M $ in Table \ref{tab:Atkin-Lehner}.
In the case when $ M $ is not a prime number nor a product of two prime numbers, that is, $ M \in \{ 4, 8, 9, 12, 16, 18, 25 \} $,
in Table \ref{tab:other_involution} we compile generalized Atkin-Lehner involutions for cusps $ s = m/M $ in $ X_0(M) $ which is not an image of $ i\infty $ under any Atkin-Lehner involution.

\begin{table}[htb] %bは下部に Table を配置. 上部ならt, その場ならh
	\caption{Cusps except for $ i\infty $ and $ 0 $ in $ X_0(M) $ with a composite number $ M $}
	\label{tab:cusps}
	\centering
	\begin{minipage}[t]{.25\textwidth}
		\begin{tabular}{cl} 	
			\hline\noalign{\smallskip}
			$M$ & Cusps \\ 	
			\hline\noalign{\smallskip}
			4 & $ \dfrac{1}{2} $ \rule[0mm]{0mm}{7mm} \\ 		%\rule[0mm]{0mm}{7mm}は行の底を+7mm
			6 & $ \dfrac{1}{2}, \dfrac{1}{3} $ \rule[0mm]{0mm}{7mm} \\
			8 & $ \dfrac{1}{2}, \dfrac{1}{4} $ \rule[0mm]{0mm}{7mm} \\
			9 & $ \dfrac{1}{3}, \dfrac{2}{3} $ \rule[0mm]{0mm}{7mm} \\
			10 & $ \dfrac{1}{2}, \dfrac{1}{5} $ \rule[-4mm]{0mm}{11mm} \\
			\hline
		\end{tabular}
	\end{minipage}
	\begin{minipage}[t]{.25\textwidth}
		\begin{tabular}{cl} 	
			\hline\noalign{\smallskip}
			$M$ & Cusps \\ 	
			\hline\noalign{\smallskip}
			12 & $ \dfrac{1}{2}, \dfrac{1}{3}, \dfrac{1}{4}, \dfrac{1}{6} $ \rule[0mm]{0mm}{7mm} \\
			16 & $ \dfrac{1}{2}, \dfrac{1}{4}, \dfrac{3}{4}, \dfrac{1}{8} $ \rule[0mm]{0mm}{7mm} \\
			18 & $ \dfrac{1}{2}, \dfrac{1}{3}, \dfrac{2}{3}, \dfrac{1}{6}, \dfrac{5}{6}, \dfrac{1}{9} $ \rule[0mm]{0mm}{7mm} \\
			25 & $ \dfrac{1}{5}, \dfrac{2}{5}, \dfrac{3}{5}, \dfrac{4}{5} $ \rule[-4mm]{0mm}{11mm} \\ %行の中心を4㎜下にして底を+11mm 
			& \rule[0mm]{0mm}{8mm} \\ 		
			\hline\noalign{\smallskip}
		\end{tabular}
	\end{minipage}
\end{table}

\begin{table}[htb] %bは下部に Table を配置. 上部ならt, その場ならh
	\caption{Atkin-Lehner involutions for $ m \neq 1, M $}
	\label{tab:Atkin-Lehner}
	\centering
	\begin{tabular}{cll} 	
		\hline\noalign{\smallskip}
		$M$ & A cusp $ s = m/M $ & 
		An Atkin-Lehner involution $ W_m = W_m^M $ \\ 	
		\hline\noalign{\smallskip}
		6 & $ \dfrac{1}{2} = \dfrac{3}{6}, \, \dfrac{1}{3} = \dfrac{2}{6} $ &
		$ W_3 = \dfrac{1}{\sqrt{3}} 
		\begin{pmatrix} 3 & 1 \\ 6 & 3 \end{pmatrix}, \quad  
		W_4 = \dfrac{1}{\sqrt{2}} 
		\begin{pmatrix} 2 & 1 \\ 6 & 4 \end{pmatrix} $  \\
		10 & $ \dfrac{1}{2} = \dfrac{5}{10}, \, \dfrac{1}{5} = \dfrac{2}{10} $ &
		$ W_5 = \dfrac{1}{\sqrt{5}} 
		\begin{pmatrix} 5 & 2 \\ 10 & 5 \end{pmatrix}, \quad  
		W_2 = \dfrac{1}{\sqrt{2}} 
		\begin{pmatrix} 2 & 1 \\ 10 & 6 \end{pmatrix} $  \\
		12 & $ \dfrac{1}{3} = \dfrac{4}{12}, \, \dfrac{1}{4} = \dfrac{3}{12} $ &
		$ W_4 = \dfrac{1}{2} 
		\begin{pmatrix} 4 & 1 \\ 12 & 4 \end{pmatrix}, \quad  
		W_3 = \dfrac{1}{\sqrt{3}} 
		\begin{pmatrix} 3 & 2 \\ 12 & 9 \end{pmatrix} $  \\
		18 & $ \dfrac{1}{2} = \dfrac{9}{18}, \, \dfrac{1}{9} = \dfrac{2}{18} $ &
		$ W_9 = \dfrac{1}{3} 
		\begin{pmatrix} 9 & 4 \\ 18 & 9 \end{pmatrix}, \quad   
		W_2 = \dfrac{1}{\sqrt{2}} 
		\begin{pmatrix} 2 & 1 \\ 18 & 10 \end{pmatrix}$ \\	
		\hline\noalign{\smallskip}
	\end{tabular}
\end{table}

\begin{table}[htb] %bは下部に Table を配置. 上部ならt, その場ならh
	\caption{Generalized Atkin-Lehner involutions for cusps $ \notin \{ 0, \infty, \frac{1}{p^e}, \frac{1}{q^f} \} $ when $M=p^e q^f$}
	\label{tab:other_involution}
	\centering
	\begin{tabular}{cll} 
		\hline\noalign{\smallskip}
		$M$ & A cusp $ s = m/M $ 
		& A generalized Atkin-Lehner involution $ W_m = W_m^M $ \\ 		
		\hline\noalign{\smallskip}
		4 & $ \dfrac{1}{2} = \dfrac{2}{4} $ & 
		$ W_2 = \dfrac{1}{2} 
		\begin{pmatrix} 2 & 2 \\ 4 & 6 \end{pmatrix} $ \\
		8 & $ \dfrac{1}{2} = \dfrac{4}{8}, \, \dfrac{1}{4} = \dfrac{2}{8} $ 
		& $ W_4 = \dfrac{1}{2\sqrt{2}} 
		\begin{pmatrix} 4 & 1 \\ 8 & 4 \end{pmatrix}, \quad  
		W_2 = \dfrac{1}{2} \begin{pmatrix} 2 & 1 \\ 8 & 6 \end{pmatrix} $ \\
		9 & $ \dfrac{1}{3} = \dfrac{3}{9}, \, \dfrac{2}{3} = \dfrac{6}{9} $ 
		& $ W_3 = \dfrac{1}{3} \begin{pmatrix} 3 & 2 \\ 9 & 9 \end{pmatrix}, \quad  
		W_6 = \dfrac{1}{3} \begin{pmatrix} 6 & 1 \\ 9 & 3 \end{pmatrix} $ \\
		12 & $ \dfrac{1}{2} = \dfrac{6}{12}, \, \dfrac{1}{6} = \dfrac{2}{12} $ 
		& $ W_6 = \dfrac{1}{2\sqrt{3}} 
		\begin{pmatrix} 6 & 2 \\ 12 & 6 \end{pmatrix}, \quad  
		W_2 = \dfrac{1}{2} \begin{pmatrix} 2 & 1 \\ 12 & 8 \end{pmatrix} $ \\
		16 & $ \dfrac{1}{2} = \dfrac{8}{16}, \, \dfrac{1}{4} = \dfrac{4}{16},$   
		& $ W_8 = \dfrac{1}{4} \begin{pmatrix} 8 & 1 \\ 16 & 4 \end{pmatrix}, \quad  
		W_4 = \dfrac{1}{4} \begin{pmatrix} 4 & 1 \\ 16 & 8 \end{pmatrix},$ \\
		& $ \dfrac{3}{4} = \dfrac{12}{16}, \dfrac{1}{8} = \dfrac{2}{16} $ &
		$ W_{12} = \dfrac{1}{4} \begin{pmatrix} 12 & 2 \\ 16 & 4 \end{pmatrix}, \quad  
		W_2 = \dfrac{1}{2} \begin{pmatrix} 2 & 1 \\ 16 & 10 \end{pmatrix} $ \\
		18 & $ \dfrac{1}{3} = \dfrac{6}{18}, \, \dfrac{2}{3} = \dfrac{12}{18}, $  
		& $ W_6 = \dfrac{1}{3\sqrt{2}} 
		\begin{pmatrix} 6 & 1 \\ 18 & 6 \end{pmatrix}, \quad  
		W_{12} = \dfrac{1}{3\sqrt{2}} 
		\begin{pmatrix} 12 & 3 \\ 18 & 6 \end{pmatrix},$ \\
		& $\dfrac{1}{6} = \dfrac{3}{18}, \, \dfrac{5}{6} = \dfrac{15}{18} $  &
		$ W_3 = \dfrac{1}{3} \begin{pmatrix} 3 & 1 \\ 18 & 9 \end{pmatrix}, \quad  
		W_{15} = \dfrac{1}{3} \begin{pmatrix} 15 & 2 \\ 18 & 3 \end{pmatrix} $ \\
		25 & $ \dfrac{1}{5} = \dfrac{5}{25}, \, \dfrac{2}{5} = \dfrac{10}{25},  $
		&  $ W_5 = \dfrac{1}{5} 
		\begin{pmatrix} 5 & 1 \\ 25 & 10 \end{pmatrix}, \quad  
		W_{10} = \dfrac{1}{5} 
		\begin{pmatrix} 10 & 1 \\ 25 & 5 \end{pmatrix},$ \\
		& $ \dfrac{3}{5} = \dfrac{15}{25}, \, \dfrac{4}{5} = \dfrac{20}{25} $  &
		$ W_{15} = \dfrac{1}{5} 
		\begin{pmatrix} 15 & 2 \\ 25 & 5 \end{pmatrix}, \quad  
		W_{20} = \dfrac{1}{5} 
		\begin{pmatrix} 20 & 3 \\ 25 & 5 \end{pmatrix} $ \rule[-4mm]{0mm}{11mm} \\	
		\hline\noalign{\smallskip}
	\end{tabular}
\end{table}

% --------------------------------------------------------------------------

\section{The normalizer of $ \Gamma_0^{(M')}(M) $} \label{sec:normalizer}

% --------------------------------------------------------------------------

In this section, we study the normalizers of $ \Gamma_0(M) $, $ G_0(M) $, and the congruence subgroup $ \Gamma_0^{(M')}(M) $ in $ \SL_2(\Z) $ and give conditions whether generalized Atkin-Lehner involutions normalize them.
Throughout this section, we fix positive integers $ M, f $ and $ M_0 $ such that $ M = f^2 M_0 $ and $ M_0 $ is square-free.

Firstly, we prepare the following subgroup of $ \SL_2(\R) $ which turns out to be the normalizer of $ \Gamma_0^{(M')}(M) $ later in this section.

\begin{definition}
	For a positive integer $ h $ with $ h^2 \mid M $, that is, $ h \mid f $, define
	\begin{align*}
	\Gamma_0^{*, h}(M) &:= \left\{
	\frac{1}{\sqrt{e}} \pmat{ ep & q/h \\ Mr/h & es } \in \SL_2(\R)
	\relmiddle|
	\begin{array}{l}
	e \in \Z_{>0}, e \mid M/h^2, \\
	p, q, r, s \in \Z
	\end{array}
	\right\}, \\
	G_0^{*, h}(M) &:= \left\{
	\frac{1}{\sqrt{e}} \pmat{ ep & q/h \\ Mr/h & es } \in \SL_2(\R)
	\relmiddle|
	\begin{array}{l}
	e \in \Z_{>0}, e \mid M/h^2, \\
	p, q, r, s \in \Z_{(M)}
	\end{array}
	\right\}.
	\end{align*}
	%	Also we define $ \Gamma_0^{*}(M) := \Gamma_0^{*, 1}(M) $ and $ G_0^{*}(M) := G_0^{*, 1}(M) $.
\end{definition}

We give several remarks for $ \Gamma_0^{*, h}(M) $.

\begin{remark}
	\begin{enumerate}
		\item The symbol $ \Gamma_0^{*, h}(M) $ is introduced in \cite[Definition 1.7]{Zemel}.
		\item The subset $ \Gamma_0^{*, h}(M) $ is a subgroup of $ \SL_2(\R) $ by \cite[Proposition 1.2 (i)]{Zemel} and the same argument shows that $ G_0^{*, h}(M) $ is a subgroup of $ \SL_2(\R) $.
		\item For a fixed integer $ h $ with $ h^2 \mid M $ and a matrix 
		\[
		\frac{1}{\sqrt{e}} \pmat{ ep & q/h \\ Mr/h & es } 
		\in \Gamma_0^{*, h}(M),
		\]
		the positive integer $ e $ is uniquely determined and called eterminant in \cite{AS}.
		\item The group $ \Gamma_0^{*, 1}(M) $ is generated by $ \Gamma_0(M) $ and Atkin-Lehner involutions and it is usually written as $ \Gamma_0^*(M) $.
	\end{enumerate}
\end{remark}

The following states whether generalized Atkin-Lehner involutions are in $ \Gamma_0^{*, h}(M) $.

\begin{lemma} \label{lem:gen_A-L_in_G*}
	For a positive integer $ m $ and a generalized Atkin-Lehner involution $ W_m^M $, we have $ W_m^M \in \Gamma_0^{*, (f, m)}(M) $.
\end{lemma}

\begin{proof}
	Let
	\[
	W_m^M = \frac{1}{\sqrt{D}} \pmat{ m & u \\ M & (M, m)v }, \quad
	D := (M, m^2).
	\]
	Here we put $ e := D/(f, m)^2, m' := m/(f, m) $.
	Then we have
	\[
	W_m^M = \frac{1}{\sqrt{e}} \pmat{ m' & u/(f, m) \\ M/(f, m) & (M/(f, m), m')v }
	\]
	and $ e = (M_0 f^2/(f, m)^2, m^{\prime 2}) $.
	Since $ f^2/(f, m)^2 $ and $ m^{\prime 2} $ are coprime and $ M_0 $ is square-free, we have $ e = (M_0, m') $.
	Thus $ e \mid (M/(f, m), m') $ and $ W_m^M \in \Gamma_0^{*, (f, m)}(M) $.
\end{proof}

Here we remark that Lemma \ref{lem:gen_A-L_in_G*} does not cover the fact that Atkin-Lehner involutions are in $ \Gamma_0^{*}(M) $.

Secondly, we compare $ \Gamma_0^{*, h}(M) $ and $ \Gamma_0^{*, h'}(M) $.

\begin{lemma} \label{lem:G*_h_include}
	If $ h \mid h' $ then $ \Gamma_0^{*, h}(M) \subset \Gamma_0^{*, h'}(M) $ and $ G_0^{*, h}(M) \subset G_0^{*, h'}(M) $.
\end{lemma}

\begin{proof}
	We prove only the first statement since the second statement can be proved by the following argument.
	
	For a matrix 
	\[
	W = \frac{1}{\sqrt{e}} \pmat{ ep & q/h \\ Mr/h & es } 
	\in \Gamma_0^{*, h}(M),
	\]
	put positive integers $ g, e_0 $ such that $ e = g^2 e_0 $ and $ e_0 $ is square-free.
	Let $ g' := (g, h'/h) $ and $ e' := e/g^{\prime 2} $.
	Since
	\[
	W = \frac{1}{\sqrt{e'}} 
	\pmat{ e'g'p & q(h'/gh)/h' \\ Mr(h'/gh)/h' & e'g's },
	\]
	it suffices to show $ e' \mid M/h^{\prime 2} $.
	
	Because $ g^2 e_0 = e \mid M/h^2 = M_0 (f/h)^2 $ and $ M_0 $ is square-free, we have $ g \mid f/h $.
	Thus $ e_0 = e/g^2 \mid M/g^2 h^2 = M_0 (f/gh)^2 $.
	We have $ e_0 \mid M_0 f/gh $ since $ e_0 $ is square-free.
	Here $ \det W = 1 $ implies
	\[
	1 = \left( e, \frac{M}{h^2 e} \right) 
	= \left( g^2 e_0, \frac{M_0 f/gh}{e_0} \frac{f}{gh} \right)
	\]
	and thus $ (e_0, f/gh) = 1 $.
	Therefore we have $ e_0 \mid M_0 $.
	
	Since $ g \mid f/h $, we have
	$ g/g' \mid f/g' h = (f/h')(h'/g' h) $.
	By the definition of $ g' $, we have 
	$ (g/g', h'/g' h) = 1 $ and thus $ g/g' \mid f/h' $.
	
	As a result, we have $ e' = e_0 (g/g')^2 \mid M_0 (f/h')^2 = M/h^{\prime 2} $.
\end{proof}

\begin{lemma} \label{lem:G*_h_replace}
	For a positive integer $ h $ with $ h^2 \mid M $ and a matrix 
	\[
	W = \frac{1}{\sqrt{e}} \pmat{ ep & q/h \\ Mr/h & es } 
	\in \Gamma_0^{*, h}(M),
	\]
	let $ g := (h, p, s)(h, q, r) $ and $ h' := h/g $.
	Then we have $ W \in \Gamma_0^{*, h'}(M) $.
	
	The same result holds for $ G_0^{*, h}(M) $.
\end{lemma}

\begin{proof}
	Since $ \det W = 1 $, $ (h, p, s) $ and $ (h, q, r) $ are coprime and thus $ g \mid h $.
	Let
	\[
	e' := e (h, p, s)^2,
	p' := p/(h, p, s)^2,
	q' := q/(h, q, r)^2,
	r' := r/(h, q, r)^2,
	s' := s/(h, p, s)^2.
	\]
	Then we have
	\[
	W = \frac{1}{\sqrt{e'}} \pmat{ e'p' & q'/h' \\ Mr'/h' & e's' }
	\]
	and $ e' \mid Mg^2/h^2 = M/h^{\prime 2} $.
	Thus we have $ W \in \Gamma_0^{*, h'}(M) $.
	
	The statement for $ G_0^{*, h}(M) $ is proved by the same argument.
\end{proof}

Thirdly, we determine the normalizers.

The first statement in the following proposition is proved in \cite[Lemma 2.1, Proposition 2.3]{Zemel}.
We give other proof.

\begin{proposition} \label{prop:G*f}
	We have
	\begin{align*}
	\left\{ W \in \SL_2(\R) \relmiddle| 
	W^{-1} \Gamma_1(M) W \subset \Gamma_0(M) \right\}
	&= \Gamma_0^{*, f}(M), \\
	\left\{ W \in \SL_2(\R) \relmiddle| 
	W^{-1} \Gamma_1(M) W \subset G_0(M) \right\}
	&= G_0^{*, f}(M).
	\end{align*}
\end{proposition}

\begin{proof}
	We prove only the first statement.
	The second statement can be proved similarly by replacing $ \Z $ with $ \Z_{(M)} $ in the following argument.
	
	Let
	\[
	W = \pmat{ p & q \\ r & s }
	\]
	be an element of the left-hand side.
	Then
	\begin{align*}
	W^{-1} \pmat{ 1 & 1 \\ 0 & 1 } W &=
	\pmat{ 1+rs & s^2 \\ -r^2 & 1-rs }, \\
	W^{-1} \pmat{ 1 & 0 \\ M & 1 } W &=
	\pmat{ 1 - Mpq & -Mq^2 \\ Mp^2 & 1 + Mpq }, \\
	W^{-1} \pmat{ 1 & 1 \\ M & 1+M } W &=
	\pmat{ 1 + rs - Mq(p+r) & s^2 - Mq(q+s) \\ -r^2 + Mp(p+r) & 1 - rs + Mp(q+s) }
	\end{align*}
	are in $ \Gamma_0(M) $.
	Hence $ p^2, Mq^2, r^2/M, s^2, Mpq, pr, rs \in \Z $.
	Therefore we can rewrite
	\[
	W = \frac{1}{\sqrt{e}} \pmat{ ep & q/f \\ Mr/f & es } 
	\]
	with a positive square-free integer $ e $ and rational numbers $ p, q, r, s $ which satisfy
	\[
	ep^2, \frac{M}{f^2 e} q^2, \frac{M}{f^2 e} r^2, es^2 \in \Z.
	\]
	Since $ e $ and $ M_0 = M/f^2 $ are square-free, $ p, q, r, s $ are integers.
	Also we have $ Mqr / f^2 e = 1 - eps \in \Z $.
	Thus $ e \mid M(q^2, qr, r^2)/f^2 = M_0 (q, r)^2 $.
	Since $ M_0 $ is square-free, we have $ e \mid M_0 (q, r) $.
	
	Let $ g := (e, q, r) $.
	We will show $ g = 1 $.
	Since $ \det W = 1 $, we have 
	$ e = (e^2, M(q, r)/f^2) = (e^2, M_0 g (q, r)/g) $.
	Here $ (q, r)/g $ is coprime to $ e $ and thus 
	$  e = (e^2, M_0 g ) = g (g(e/g)^2, M_0) $.
	Since $ M_0 $ is square-free, we have $ e = g (e, M_0) $.
	Therefore $ e/g \mid M_0 $ and $ (g, M_0 / (e/g)) = 1 $.
	Hence $ g M_0 / (e/g) = g^2 M_0 / e $ is square-free.
	Since $ e $ is square-free, we have $ g=1 $.
	
	By combining the above discussion, we have $ e \mid M_0 $ and thus $ W \in \Gamma_0^{*, f}(M) $.
	
	Conversely, suppose 
	\[
	W = \frac{1}{\sqrt{e}} \pmat{ ep & q/f \\ Mr/f & es } 
	\in \Gamma_0^{*, f}(M).
	\]
	For any matrix
	$ \gamma = \bigl( \begin{smallmatrix} a & b \\ Mc & d \end{smallmatrix} \bigr) \in \Gamma_1(M) $,
	we have
	\begin{equation} \label{eq:conjugate}
	\begin{split}
	W^{-1} \gamma W &=
	\begin{pmatrix}	A &	B \\ MC & D \end{pmatrix}, \\
	A &:= eaps + \frac{M}{f} \left( brs - cpq \right) - \frac{M}{f^2 e}dqr , \\
	B &:= -\frac{(a-d)qs}{f} - \frac{M}{f^2 e} cq^2 + ebp^2, \\
	C &:= -\frac{(a-d)pr}{M} - \frac{M}{f^2 e} br^2 + ecs^2, \\
	D &:= -\frac{M}{f^2 e} aqr - \frac{M}{f} \left( brs - cpq  \right) + edps.	
	\end{split}
	\end{equation}
	Thus we obtain $ W^{-1} \gamma W \in \Gamma_0(M) $.
\end{proof}

To determine the normalizers of $ \Gamma_0(M) $ and $ G_0(M) $, we prepare a lemma from elementary number theory.

\begin{lemma} \label{lem:e(M)}
	For a positive integer $ M $, let $ \veps(M) $ be the greatest common divisor of $ a-d $ for all integers $ a $ and $ d $ such that $ \smat{a & * \\ * & d} \in \Gamma_0(M) $.
	Similarly, let $ e(M) $ be the greatest common divisor of $ a-d $ for all integers $ a $ and $ d $ such that $ \smat{a & * \\ * & d} \in G_0(M) \cap M_2(\Z) $.
	Then we have
	\[
	\veps(M) = (M, 24), \quad e(M) = (M, 2).
	\]
	Thus we have $ \Gamma_0(M) = \Gamma_0^{(M')}(M) $ where  $ M' := (M, 24) $. 
\end{lemma}

\begin{proof}
	Since the map
	\[
	\Gamma_0(M) \to (\Z/M\Z)^\times, 
	\pmat{a & * \\ * & *} \mapsto a \bmod M 
	\]
	is surjective, we have
	\[
	\veps(M) = 
	( a-d \mid ad \equiv 1 \bmod M ), \quad  
	e(M) = 
	( a-d \mid a, d \in \Z \cap \Z_{(M)}^\times ).
	\]
	It suffices to show when $ M $ is a power of a prime number since
	$ \veps(M), e(M) $ divides $ M $. 
	
	If $ M=16 $, then a pair $ (\overline{a}, \overline{d}) \in ((\Z/16\Z)^\times)^2 $ such that $ ad \equiv 1 \bmod 16 $ is
	$ \pm (\overline{1}, \overline{1}) $, $ \pm (\overline{3}, \overline{11}) $ or $ \pm (\overline{7}, \overline{7}) $.
	This implies that $ \veps(16) = (16, 1-1, 11-3, 7-7) = 8 $. 
	Thus if $ M $ is a power of 2, then $ \veps(M) = (M, 24) $. 
	
	If $ M=9 $, then then a pair $ (\overline{a}, \overline{d}) \in ((\Z/9\Z)^\times)^2 $ such that $ ad \equiv 1 \bmod 9 $ is
	$ \pm (\overline{1}, \overline{1}) $ or $ \pm (\overline{2}, \overline{5}) $ and we have $ \veps(9) = 3 $. 
	Thus if $ M $ is a power of 3, then $ \veps(M) = (M, 24) $. 
	
	Suppose $ M $ is a power of a prime number $ p \ge 5 $. 
	For an integer $ a $ such that $ 2a \equiv 1 \bmod p $,
	since $ 4 \not\equiv 1 \bmod p $ we have $ a \not\equiv 2 \bmod p $. 
	Therefore, we find $ \veps(p) \mid (p, a-2) = 1 $ and $ \veps(M) = (M, 24) $. 
	
	The same argument works for $ e(M) $. 
\end{proof}

For a group $ G $ and its subgroup $ H $, we denote by $ N_G(H) $ the normalizer of $ H $ in $ G $.

By Proposition \ref{prop:G*f},  (\ref{eq:conjugate}) and Lemma \ref{lem:e(M)}, we have the following proposition whose the first statement is stated without a proof in \cite{AS}, \cite[Theorem 1]{Bars}, \cite[Section 3]{CN} and a proof is found in \cite[Corollary 3.2]{Zemel}.

%Proposition \ref{prop:gen_A-L_normalize_G} follows from Lemma \ref{lem:gen_A-L_in_G*}, Lemma \ref{lem:G*_h_include}, and the following proposition.

\begin{proposition} \label{prop:normalizer_of_G_0}
	We have the followings. 
	\begin{enumerate}
		\item \label{item:prop:normalizer_of_G_0_1}
		$ N_{\SL_2(\R)}(\Gamma_0(M)) = \Gamma_0^{*, (f, 24)}(M) $.
		\item \label{item:prop:normalizer_of_G_0_2}
		$ N_{\SL_2(\R)}(G_0(M)) = \Gamma_0^{*, (f, 2)}(M) $.
		\item \label{item:prop:normalizer_of_G_0_3}
		For a positive divisor $ M' $ of $ M $, we have
		$ N_{\SL_2(\R)}(\Gamma_0^{(M')}(M)) = \Gamma_0^{*, (f, M', 2M/M')}(M) $.	
	\end{enumerate}
\end{proposition}

\begin{proof}
	\ref{item:prop:normalizer_of_G_0_1}.	
	Since both sides are in $ \Gamma_0^{*, f}(M) $ by Lemma \ref{lem:G*_h_include} and Proposition \ref{prop:G*f}, it is enough to show that for $ W \in \Gamma_0^{*, f}(M) $, $ W $ normalizes $ \Gamma_0(M) $ if and only if $ W \in \Gamma_0^{*, (f, 24)}(M) $.
	By Lemma \ref{lem:G*_h_replace}, we can assume
	\[
	W = \frac{1}{\sqrt{e}} \pmat{ ep & q/h \\ Mr/h & es } 
	\in \Gamma_0^{*, h}(M), \quad
	h \mid f, \quad
	(h, p, s) = (h, q, r) = 1.
	\]
	In this case, $ W $ normalizes $ \Gamma_0(M) $ if and only if 
	$ h \mid (a-d)(pr, qs) $ for any $ a, d \in \Z $ with $ ad \equiv 1 \bmod M $ by  (\ref{eq:conjugate}).
	By Lemma \ref{lem:e(M)}, this is equivalent to $ h \mid (M, 24)(pr, qs) $.
	
	Since $ \det W = 1 $, we have $ (p, q) = (p, r) = (q, s) = 1 $.
	Thus
	\begin{align*}
	(p, s) &= (p, q)(p, s) =((p, q, s)p, qs) = (p, qs), \\
	(q, r) &= (q, r)(s, r) =((q, r, s)r, qs) = (r, qs), \\
	(p, qs)(r, qs) &= (pr, (p, r, qs)qs) = (pr, qs).
	\end{align*}
	and we obtain $ (pr, qs) = (p, s)(q, r) $.
	Since $ (h, p, s) = (h, q, r) = 1 $, $ W $ normalizes $ \Gamma_0(M) $ if and only if $ h \mid (M, 24) $.
	Since $ f \mid h $, this is equivalent to $ h \mid (f, 24) $, that is, $ W \in \Gamma_0^{*, (f, 24)}(M) $ by Lemma \ref{lem:G*_h_include}.
	
	\ref{item:prop:normalizer_of_G_0_2} can be proved by the same argument in 	\ref{item:prop:normalizer_of_G_0_1}.
	
	\ref{item:prop:normalizer_of_G_0_3}.
	Let
	\[
	W = \frac{1}{\sqrt{e}} \pmat{ ep & q/h \\ Mr/h & es } 
	\in \Gamma_0^{*, h}(M)
	\]
	with $ h \mid f, (h, p, s) = (h, q, r) = 1 $.
	For a matrix
	$ \gamma = \bigl( \begin{smallmatrix} a & b \\ Mc & d \end{smallmatrix} \bigr) \in \Gamma_0^{(M')}(M) $,
	let
	\[
	W^{-1} \gamma W =
	\begin{pmatrix}	A &	B \\ MC & D \end{pmatrix}.
	\]
	
	By  (\ref{eq:conjugate}), $ B, C \in \Z $ for any $ \gamma \in \Gamma_0^{(M')}(M) $ if and only if
	$ h \mid M'(pr, qs) = M'(p, s)(q, r) $.
	Since $ (h, p, s) = (h, q, r) = 1 $ and $ f \mid h $, this is equivalent to $ h \mid (f, M') $.
	
	By  (\ref{eq:conjugate}), $ A \equiv D \bmod M' $ for any $ \gamma \in \Gamma_0^{(M')}(M) $ if and only if
	\[
	0 \equiv A - D \equiv 2 \frac{M}{h} (brs - cpq) \bmod M'
	\]
	for any $ \gamma \in \Gamma_0^{(M')}(M) $, which is equivalent to 
	$ h \mid (pr, qs)2M/M' = (p, s)(q, r)2M/M' $.
	This is equivalent to $ h \mid (f, 2M/M') $.
\end{proof}

Finally, we give conditions whether generalized Atkin-Lehner involutions normalize $ \Gamma_0(M) $, $ G_0(M) $, or $ \Gamma_0^{(M')}(M) $.

\begin{proposition} \label{prop:gen_A-L_normalize_G}
	For a positive integer $ m $ and a generalized Atkin-Lehner involution $ W_m^M $, the followings hold.
	\begin{enumerate}
		\item \label{item:prop:gen_A-L_normalize_G1}
		$ W_m^M \in N_{\SL_2(\R)}(\Gamma_0(M)) $ if and only if 
		$ (f, m) \mid (f, 24) $.
		\item \label{item:prop:gen_A-L_normalize_G2}
		$ W_m^M \in N_{\SL_2(\R)}(G_0(M)) $ if and only if 
		$ (f, m) \mid (f, 2) $.
		\item \label{item:prop:gen_A-L_normalize_G3}
		For $ M' := M/(f, m) $, we have 
		$ W_m^M \in N_{\SL_2(\R)}(\Gamma_0^{(M')}(M)) $.
	\end{enumerate}
\end{proposition}

\begin{proof}
	These follow from Lemma \ref{lem:gen_A-L_in_G*}, Lemma \ref{lem:G*_h_include}, and Proposition \ref{prop:normalizer_of_G_0}.
\end{proof}

%In these cases, we summarize Proposition \ref{prop:C'_0(M)_str} in the following proposition.

In the case when the modular curve $ X_0(M) $ has genus zero, we have the following.

\begin{proposition} \label{prop:inv_rekkyo}
	Let $ s \in X_0(M) $ be a cusp expressed as 
	$ s = m/M = W_m^M(i \infty) $ with an integer $ 0 \le m < M $.
	Let $ W_m^M $ be a generalized Atkin-Lehner involution.
	\begin{enumerate}
		\item If $ (M, s) \notin \{ (25, 1/5), (25, 2/5), (25, 3/5), (25, 4/5) \} $, then we have $ W_m^M \in N_{\SL_2(\R)}(\Gamma_0(M)) $.
		\item If
		\[
		(M, s) \notin
		\left\{
		\begin{gathered}
		(9, 1/3), (9, 2/3), \\
		(16, 1/2), (16, 1/4), (16, 3/4), (16, 1/8), \\
		(18, 1/3), (18, 2/3), (18, 1/6), (18, 5/6), \\
		(25, 1/5), (25, 2/5), (25, 3/5), (25, 4/5),	
		\end{gathered}
		\right\},
		\]
		then we have $ W_m^M \in N_{\SL_2(\R)}(G_0(M)) $.
		\item If $ M=25 $ and $ s \in \{ 1/5, 2/5, 3/5, 4/5 \} $, then we have
		$ W_m^M \in N_{\SL_2(\R)}(\Gamma_0^{(5)}(25)) $. 
	\end{enumerate}
\end{proposition}

% --------------------------------------------------------------------------

\section{The action of $ G_0(M) $ on cusps} \label{sec:action_on_cusps}

% --------------------------------------------------------------------------

In this section, we study the action of $ G_0(M) $ on cusps to give a condition that a pair of cusps is a point of modular correspondence and  calculate the intersection multiplicity.

Firstly, we describe explicitly cusps and the action of $ G_0(M) $ on them.

\begin{definition}
	Let 
	$ C_0(M) := \Gamma_0(M) \backslash (\Q \cup \{ i\infty \}) $
	be the set of cusps in the modular curve $ X_0(M) $ and we define its subset
	\[
	C_0(M)_n :=
	\left\{ \Gamma_0(M) \frac{l}{n} \relmiddle| \bar{l} \in (\Z/(n,M/n)\Z)^{\times} \right\}
	\]
	for a positive divisor $ n $ of $ M $.
\end{definition}

We have a decomposition
\begin{equation} \label{eq:cusp_rekkyo}
	C_0(M) = \coprod_{n \mid M} C_0(M)_n 
\end{equation}
by the following lemma.

\begin{lemma}[{\cite[Proposition 3.8.3]{DS}}] \label{lem:cusp_criterion}
	For integers $ l, n, l', n' $ with $ (l, n) = (l', n') = 1 $,
	let $ s = l/n, s' = l'/n' $.
	Then $ \Gamma_0(M) s = \Gamma_0(M) s' $ if and only if
	there exists an integer $ d $ such that 
	\[
	(d', M) = 1, \quad
	n' \equiv dn \bmod M, \quad
	dl' \equiv l \bmod (M, n).
	\]
\end{lemma}

By the above lemma, the group $ G_0(M) $ acts on $ C_0(M)_n $ for a positive divisor $ n $ of $ M $.
This action is essentially it of $ G_0(M)_{i \infty} $ by Proposition \ref{prop:str_G_0(M)}.
We describe it in the following Proposition.

\begin{proposition} \label{prop:s=A(s')}
	Let $ 0 \le m, m' < M $ be integers, $ s = m/M, s' = m'/M $ be rational numbers and
	\[
	A = 
	\begin{pmatrix} a & b \\ 0 & d \end{pmatrix}
	\in G_0(M) \cap M_2(\Z).
	\]
	Then $ \Gamma_0(M) s = \Gamma_0(M) A(s') $ if and only if
	\[
	(M, m) = (M, m'), \quad m'ad \equiv m(d, m'a + Mb)^2 \bmod M.
	\]		
\end{proposition}

\begin{proof}
	Let $ g := (d, m'a + Mb) $.
	By Lemma \ref{lem:cusp_criterion}, $ \Gamma_0(M) s = \Gamma_0(M) A(s') $ if and only if	there exists an integer $ d' $ such that 
	\[
	(d', M) = 1, \quad
	\frac{M}{(M, m')} \frac{d}{g} \equiv d' \frac{M}{(M, m)} \bmod M, \quad
	d' \frac{m'a + Mb}{(M, m')g} \equiv \frac{m}{(M, m)} \bmod \frac{M}{(M, m)}.
	\]
	This is equivalent to the condition in the statement since we can choose 
	$ d' = d/g $.	
\end{proof}

Secondly, we establish definitions of some kind of cusps and study the action of $ G_0(M) $ on them.

\begin{definition}
	We define subsets of $ C_0(M) $ as
	\begin{align*}
		C'_0(M) &:=
		\{ W(i\infty) \in C_0(M) \mid 
		W \in N_{\SL_2(\R)}(\Gamma_0(M)) \}, \\
		C^{\prime \prime}_0(M) &:=
		\{ W(i\infty) \in C_0(M) \mid 
		W \in N_{\SL_2(\R)}(\Gamma_0(M)) \cap N_{\SL_2(\R)}(G_0(M)) \}.	
	\end{align*}
\end{definition}

%In section \ref{sec:normalizer}, we give a decomposition of $ C'_0(M) $ and $ C^{\prime \prime}_0(M) $.
%The following Proposition follows from Proposition \ref{prop:gen_A-L_normalize_G}.

\begin{proposition} \label{prop:C'_0(M)_str}
	It holds that
	\[
	C'_0(M) = \coprod_{n \mid M, \, (n, M/n) \mid (24, M/(M, 24))} C_0(M)_n, \quad
	C''_0(M) = \coprod_{n \mid M, \, (n, M/n) \mid (2, M/(M, 2))} C_0(M)_n.	
	\]
\end{proposition}

\begin{proof}
	Let $ f $ and $ M_0 $ be positive integers such that $ M = f^2 M_0 $ and $ M_0 $ is square-free.
	For a positive integer $ m $, let $ n := M/(M, m) $.
	It suffices to show that $ (f, n) = (n, M/n) $ since $ (f, m) = (f, n) $.
	Since $ (n, M/n)^2 \mid n \cdot M/n = M $, we have $ (n, M/n) \mid (f, n) $.
	Since 
	\[
	(f, n)^2 \mid M = \frac{M}{n} \frac{n}{(f, n)} (f, n),
	\]
	we have $ (f, n) \mid M/n $ and thus $ (f, n) \mid (n, M/n) $.
	By Proposition \ref{prop:gen_A-L_normalize_G}, we obtain the statement.
\end{proof}

The actions of $ G_0(M) $ on $ C'_0(M) $ and $ C''_0(M) $ is described as follows by Proposition \ref{prop:C'_0(M)_str}.

\begin{corollary}
	The sets p--00--$ C''_0(M), C'_0(M) \setminus C''_0(M) $, and $ C_0(M) \setminus C'_0(M) $ are stable under the action of $ G_0(M) $.
\end{corollary}

In the case when the modular curve $ X_0(M) $ has genus zero, we have the following by Proposition \ref{prop:inv_rekkyo}.

\begin{corollary} \label{cor:aut_rekkyo}
	Let $ 1 \le M \le 10 $ or $ M \in \{ 12, 13, 16, 18, 25 \} $.
	\begin{enumerate}
		\item If $ M \neq 25 $, then $ C_0(M) = C'_0(M) $.
		\item If $ M \notin \{ 9, 16, 18, 25 \} $, then $ C_0(M) = C'_0(M) = C^{\prime \prime}_0(M) $.
		\item If $ M \in \{ 9, 16, 18, 25 \} $, then
		\begin{alignat*}{2}
			C_0(9) \smallsetminus C^{\prime \prime}_0(9) &= 
			\left\{ \dfrac{1}{3}, \dfrac{2}{3} \right\}, \quad
			&
			C^{\prime \prime}_0(9) &= 
			\left\{ i\infty, 0 \right\}, 
			\\
			C_0(16) \smallsetminus C^{\prime \prime}_0(16) &= 
			\left\{ \dfrac{1}{2}, \dfrac{1}{4}, \dfrac{3}{4}, \dfrac{1}{8} \right\}, \quad
			&	
			C^{\prime \prime}_0(16) &= 
			\left\{ i\infty, 0 \right\}, 
			\\
			C_0(18) \smallsetminus C^{\prime \prime}_0(18) &= 
			\left\{ \dfrac{1}{3}, \dfrac{2}{3}, \dfrac{1}{6}, \dfrac{5}{6} \right\}, \quad
			&
			C^{\prime \prime}_0(18) &= 
			\left\{ i\infty, 0, \dfrac{1}{2}, \dfrac{1}{9} \right\}, 
			\\	
			C_0(25) \smallsetminus C^{\prime \prime}_0(25) &= 
			\left\{ \dfrac{1}{5}, \dfrac{2}{5}, \dfrac{3}{5}, \dfrac{4}{5} \right\}, \quad
			&
			C'_0(25) &= C^{\prime \prime}_0(25) = 
			\left\{ i \infty, 0 \right\}.
		\end{alignat*}
	\end{enumerate}
\end{corollary}

% --------------------------------------------------------------------------

\section{Intersection multiplicities at cusps} \label{sec:mult_cusp}

% --------------------------------------------------------------------------

In the rest of this paper, we assume that the modular curve $ X_0(M) $ has genus zero, that is, $ 1 \le M \le 10 $ or $ M= 12, 13, 16, 18, 25 $.

Our goal in this section is to calculate the intersection multiplicity of the modular correspondences at a pair $ (s, s^\prime) $ of cusps in the modular curve $ X_0(M) $. 

Firstly, we consider the condition $ \Gamma_0(M) s = \Gamma_0(M) A(s') $
for $ s, s' \in \Q \cup \{ i \infty \} $ and a matrix 
$ A \in G_0(M) $ in  (\ref{eq:def_of_T}).

Let $ t \colon X_0(M) \xrightarrow{\sim} \bbP^1(\C) $ be the isomorphism defined in Section \ref{sec:known}.

\begin{proposition} \label{prop:order_t_W}
	Let $ 0 \le m, m' < M $ be integers and put
	\[
	s = m/M, \quad s' = m' / M, \quad 
	D := (M, m^2), \quad D' := (M, m^{\prime 2}).
	\]
	Let
	\[
	W = W_m^M = \frac{1}{\sqrt{D}} 
	\begin{pmatrix} m & u \\ M & v \end{pmatrix}, \quad
	W' = W_{m'}^{M} = \frac{1}{\sqrt{D'}} 
	\begin{pmatrix} m' & * \\ M & * \end{pmatrix}
	\in \SL_2(\R)
	\]
	be generalized Atkin-Lehner involutions and
	\[
	A = 
	\begin{pmatrix} a & b \\ 0 & d \end{pmatrix}
	\in G_0(M) \cap M_2(\Z)
	\]
	be a matrix such that 
	$ \Gamma_0(M) s = \Gamma_0(M) A(s') $.
	Then the order of
	$ t \circ AW' (\tau) - t(s) $ with respect to $ q := e^{2 \pi \iu \tau} $ is $ (d, m'a + Mb)^2/ad $. 
\end{proposition}

\begin{proof}
	Put 
	\[	
	W^{-1} AW'
	= \sqrt{\frac{D}{D'}} \begin{pmatrix} k & * \\ -Ml & * \end{pmatrix}
	\]
	with rational numbers $ k, l $.
	Then we have
	\[
	k = \frac{v}{D} \left( m'a + Mb \right) - \frac{M}{D} u d \in \Z, \quad  
	l = \frac{m'a - md}{D} + \frac{M}{D} b \in \frac{1}{D} \Z.
	\]
	There exists a matrix
	\[
	\gamma = \pmat{ k & * \\ -Ml & * } \in \SL_2(\Z).
	\]
	Let $ g = (k, Ml) $.
	We have
	\[
	W^{-1} AW'
	= \sqrt{\frac{D}{D'}} \gamma 
	\begin{pmatrix} g & * \\ 0 & ad/g \end{pmatrix}.
	\]
	
	We show $ g = (d, m'a + Mb) $. 
	Since $ k - vl = d $ by direct calculation, we have $ g = (k, Md, Ml) $.
	Since $ k \equiv vm'a/D \bmod (M, m) $, we have $ (k, M) = 1 $.
	Thus we have 
	\[
	g = (k, d, Ml) = (k, d, Dl) = (k, d, m'a + Mb) = (d, m'a + Mb).
	\]
	
	By assumption there exists a matrix $ \delta \in \Gamma_0(M) $ such that
	$ \delta(s) = A(s') $.
	Since
	\[
	\delta W(i\infty) = \delta(s) = A(s') = AW'(i\infty)
	= W \gamma 
	\begin{pmatrix} g & * \\ 0 & ad/g \end{pmatrix} (i\infty)
	= W \gamma (i\infty),
	\]
	we have $ \gamma^{-1} W^{-1} \delta W \in \SL_2(\Q)_{i\infty} $.
	
	We show that the diagonal elements of $ \gamma^{-1} W^{-1} \delta W $ are integers. 
	If $ W \in N_{\SL_2(\R)}(\Gamma_0(M)) $, then $ W^{-1} \delta W \in \Gamma_0(M) $ and thus $ \gamma^{-1} W^{-1} \delta W \in \SL_2(\Z) $.	
	If $ W \not\in N_{\SL_2(\R)}(\Gamma_0(M)) $, then $ M = 25 $ and $ W \in N_{\SL_2(\R)}(\Gamma_0^{(5)}(25)) $ by Proposition \ref{prop:inv_rekkyo}.	
	By Definition \ref{def:other_aut} and  (\ref{eq:conjugate}), a matrix $ W^{-1} \delta W $ has integral element except for the $ (1, 2) $ entry and its $ (1, 2) $ entry is in $ 5^{-1} \Z $.
	Since $ l \in 5^{-1} \Z $, the diagonal elements of $ \gamma^{-1} W^{-1} \delta W $ are integers. 
	
	Thus we can express as
	\[
	\gamma^{-1} W^{-1} \delta W 
	= \pm \begin{pmatrix} 1 & * \\ 0 & 1 \end{pmatrix}.
	\]
	Therefore we have
	\[
	t \circ AW'
	= t \circ W \gamma 
	\begin{pmatrix} g & * \\ 0 & ad/g \end{pmatrix}
	= t \circ \delta W 
	\begin{pmatrix} \pm 1 & * \\ 0 & \pm 1 \end{pmatrix}
	\begin{pmatrix} g & * \\ 0 & ad/g \end{pmatrix}
	= t \circ W \begin{pmatrix} g & * \\ 0 & ad/g \end{pmatrix}.
	\]
	Let $ n $ be the order of $ t \circ W (\tau) - t(s) $ with respect to $ q $.
	Then the order of $ t \circ AW' (\tau) - t(s) $ with respect to $ q $ is $ ng^2/ad $. 
	We need to show $ n=1 $.
	
	If $ W \in N_{\SL_2(\R)}(\Gamma_0(M)) $, then by Proposition \ref{prop:inv_rekkyo} and thus $ t \circ W \colon X_0(M) \to \bbP^1(\C) $ is an isomorphism. 
	Therefore $ n=1 $. 
	
	If $ W \not\in N_{\SL_2(\R)}(\Gamma_0(M)) $, then $ M = 25 $ and $ W \in N_{\SL_2(\R)}(\Gamma_0^{(5)}(25)) $ by Proposition \ref{prop:inv_rekkyo}.
	The map $ t \circ W \colon X_0^{(5)}(25) \to \bbP^1(\C) $ is the composition of the isomorphism $ W \colon X_0^{(5)}(25) \to X_0^{(5)}(25) $, natural projection $ X_0^{(5)}(25) \to X_0(25) $ and the isomorphism 
	$ t \colon X_0(25) \to \bbP^1(\C) $.
	The ramification index at $ i\infty $ the natural projection $ X_0^{(5)}(25) \to X_0(25) $ is
	\[
	[ \{ \pm I \} \Gamma_0(25)_{i\infty} : \{ \pm I \} \Gamma_0^{(5)}(25)_{i\infty} ]
	= 1
	\]
	by \cite[Section 3.1]{DS} and thus we also have $ n=1 $ in this case. 
	This completes the proof. 
\end{proof}

Secondly, we consider the condition that a pair of cusps is a point on the modular correspondence.

\begin{definition}
	Let $ s, s' \in X_0(M) $ be cusps and
	\[
	W = \frac{1}{\sqrt{D}} 
	\begin{pmatrix} m & u \\ M & v \end{pmatrix},
	W' = \frac{1}{\sqrt{D'}} 
	\begin{pmatrix} m' & u' \\ M & v' \end{pmatrix}
	\in \SL_2(\R)
	\]
	be generalized Atkin-Lehner involutions such that 
	$ s = W(i \infty), s' = W'(i \infty) $.
	Let $ N $ be a positive integer coprime to $ M $. 
	We define 
	$ \delta_{s, s'}(N) := 1 $ if $ (M, m) = (M, m') $ and there exists an integer $ g $ such that $  m'N \equiv mg^2 \bmod M $.
	Otherwise we define $ \delta_{s, s'}(N) := 0 $.
\end{definition}

\begin{remark}
	In the case when $ M \neq 25 $, we have $ \delta_{s, s'}(N) = 1 $ if and only if $ D = D' $ and $ m \equiv N m' \bmod (M, m^2) $ since $ (M, m^2) \mid 12 $ and $ \bar{1} $ is the unique square element of $ (\Z/12\Z)^\times $.
	In the case when $ M = 25 $, we have $ \delta_{s, s'}(N) = 1 $ if and only if $ D = D' $ and $ s \equiv \pm N s' \bmod \Z $ by Table \ref{tab:other_involution}.
	Thus $ \delta_{s, s'}(N) = 1 $ if and only if
	$ s = s' \in C''_0(M) $,
	\begin{alignat*}{3}
		&M=9, \quad &  &N \equiv 1 \bmod 3, \quad &  &s = s' \in \left\{ \frac{1}{3}, \frac{2}{3} \right\}, \\
		&M=9, \quad &  &N \equiv 1 \bmod 3, \quad &  &\left(s, s'\right) \in \left\{ \left(\frac{1}{3}, \frac{2}{3}\right), \left(\frac{2}{3}, \frac{1}{3}\right) \right\}, \\
		&M=16, \quad & &  & &s = s' \in \left\{ \frac{1}{2}, \frac{1}{8} \right\}, \\
		&M=16, \quad   & &N \equiv 1 \bmod 4, \quad   & &s = s' \in \left\{ \frac{1}{2}, \frac{1}{4}, \frac{3}{4}, \frac{1}{8} \right\}, \\
		&M=16, \quad   & &N \equiv -1 \bmod 4, \quad   & &\left(s, s'\right) \in \left\{ \left(\frac{1}{4}, \frac{3}{4}\right), \left(\frac{3}{4}, \frac{1}{4}\right) \right\}, \\
		&M=18, \quad   & &N \equiv 1 \bmod 6, \quad   & &s = s' \in \left\{ \frac{1}{3}, \frac{2}{3}, \frac{1}{6}, \frac{5}{6} \right\}, \\
		&M=18, \quad & &N \equiv -1 \bmod 6, \quad   
		& &\left(s, s'\right) \in \left\{ \left(\frac{1}{3}, \frac{2}{3}\right), \left(\frac{2}{3}, \frac{1}{3}\right),  \left(\frac{1}{6}, \frac{5}{6}\right), \left(\frac{5}{6}, \frac{1}{6}\right) \right\}, \\
		&M=25, \quad & &s, s' \in \left\{ \frac{1}{5}, \frac{2}{5}, \frac{3}{5}, \frac{4}{5} \right\}, \quad & &s \equiv \pm N s' \bmod \Z 
	\end{alignat*}
	by Table \ref{tab:other_involution}.
\end{remark}

\begin{theorem} \label{thm:main_cusp_intersection}
	For a positive integer $ N $ coprime to $ M $, 
	the modular correspondence $ T_{N}^{\Gamma_0(M)} \subset X_0(M) \times X_0(M) $ satisfies that
	\[
	T_{N}^{\Gamma_0(M)} \subset Y_0(M)^2 \cup
	\{ (s, s') \mid \delta_{s, s'}(N) \neq 0 \}.
	\]	
	In particular, if $ M \notin \{ 9, 16, 18, 25 \} $ then a pair of cusps on $ T_{N}^{\Gamma_0(M)} $ is a form $ (s, s) $. 
\end{theorem}

\begin{proof}
	By  (\ref{eq:def_of_T}), if $ (\tau, \tau^\prime) \in X_0(M) \times X_0(M) $ is a point of $ T_{N}^{\Gamma_0(M)} $, then there exist integers $ a, b $, and $ d $ such that $ a d = N, 0 \le b < d $ and
	$ \Gamma_0(M) \tau = \Gamma_0(M) \frac{a \tau^\prime + b}{d} $. 
	Thus $ (\tau, \tau^\prime) $ is a point on $ Y_0(M) \times Y_0(M) $ or a pair of two cusps $ (s, s^\prime) $.
	
	Suppose that $ (\tau, \tau^\prime) = (s, s^\prime) $ is a pair of two cusps.
	Since $ N $ is coprime to $ M $, we have $ A := \smat{a & b \\ 0 & d} \in G_0(M) $. 
	By Proposition \ref{prop:s=A(s')}, we have
	$ (M, m) = (M, m') $ and $  m'N \equiv m(d, m'a + Mb)^2 \bmod M $
	and thus $ \delta_{s, s'}(N) = 1 $. 
\end{proof}

Finally, we calculate the intersection multiplicity at cusps by Proposition \ref{prop:order_t_W}. 

\begin{proposition} \label{prop:cusp_int_mult}
	Let $ N_1, N_2 $ be positive integers coprime to $ M $.
	Suppose that $N_1 N_2$ is not a square. 
	Then for two cusps $ s, s' $ in $ X_0(M) $, we have
	\[
	(T_{N_1}^{\Gamma_0(M)} \cdot T_{N_2}^{\Gamma_0(M)})_{(s, s')}
	= \delta_{s, s'}(N_1) \delta_{s, s'}(N_2) 
	\sum_{a_1 d_1 = N_1, a_2 d_2 = N_2}  
	\min \left\{ a_1 d_2, a_2 d_1 \right\}
	\]
	unless $ M = 25 $ and $ s, s' \in \{ 1/5, 2/5, 3/5, 4/5 \} $.
	If $ M = 25 $ and $ s, s' \in \{ 1/5, 2/5, 3/5, 4/5 \} $, then we have
	\[
	(T_{N_1}^{\Gamma_0(M)} \cdot T_{N_2}^{\Gamma_0(M)})_{(s, s')}
	= \delta_{s, s'}(N_1) \delta_{s, s'}(N_2) 
	\sum_{\substack{
		a_1 d_1 = N_1, a_2 d_2 = N_2, \\
		a_1 s \equiv d_1 s', a_2 s \equiv d_2 s' \bmod \Z
		}}  
	\min \left\{ a_1 d_2, a_2 d_1 \right\}.
	\]
\end{proposition}

\begin{proof}
	Let $ s = m/M, s' = m' / M' $ with integers $ 0 \le m, m' < M $ and  
	$ W := W_m^M, W' := W_{m'}^{M} \in \SL_2(\R) $ be generalized Atkin-Lehner involutions.	
	Let $ \Phi_{N_i}^{\Gamma_0(M)}(X, Y) \in \Z[X, Y] $ be the modular polynomial whose existence is guaranteed by Theorem \ref{thm:main_mod_poly}.
	Then the intersection multiplicity at $ (s, s') $ is
	\begin{align*}
		&( T_{N_1}^{\Gamma_0(M)} \cdot T_{N_2}^{\Gamma_0(M)})_{(s, s')} \\
		= &( T_{N_1}^{\Gamma_0(M)} \cdot T_{N_2}^{\Gamma_0(M)})
		_{(W(\infty), W'(\infty))} \\
		= &\dim_\C \C\left[\left[q, q'\right]\right] / 
		( \Phi_{N_i}^{\Gamma_0(M)}(t \circ W (\tau), t \circ W'(\tau'))
		\mid i = 1, 2 ) \\
		=&\frac{1}{N_1 N_2} 
		\sum_{
			\substack{
				A_i \in I^{\Gamma_0(M)}_{N_i, \mathrm{mat}}, \\
				\Gamma_0(M) s = \Gamma_0(M) A_i (s'), i=1, 2 
			}}  
		\dim_\C \C\left[\left[ q, q'_{N_1 N_2} \right]\right] \bigg/ 
		\left( t \circ W (\tau), t \circ A_i W'(\tau'))
		\mid i = 1, 2 \right).	
	\end{align*}
	By Proposition \ref{prop:order_t_W} in the case when $ A $ is the identity matrix, the order of
	$ t \circ W (\tau) - t(s) $ with respect to $ q  $ is $ 1 $. 
	Thus we have
	\[
	\C\left[\left[ q, q'_{N_1 N_2} \right]\right] 
	= \C\left[\left[ t \circ W (\tau), q'_{N_1 N_2} \right]\right].
	\]
	Hence the intersection multiplicity is
	\begin{align*}
	&( T_{N_1}^{\Gamma_0(M)} \cdot T_{N_2}^{\Gamma_0(M)})_{(s, s')} \\
	=&\frac{1}{N_1 N_2} 
	\sum_{
		\substack{
			A_i \in I^{\Gamma_0(M)}_{N_i, \mathrm{mat}}, \\
			\Gamma_0(M) s = \Gamma_0(M) A_i (s'), i=1, 2 
	}}  
	\dim_\C \C\left[\left[ q'_{N_1 N_2} \right]\right] \bigg/ 
	\left( t \circ A_1 W'(\tau') - t \circ A_2 W'(\tau') \right) \\
	= &\sum_{ 
		\substack{
			A_i \in I^{\Gamma_0(M)}_{N_i, \mathrm{mat}}, \\
			\Gamma_0(M) s = \Gamma_0(M) A_i (s'), i=1, 2 
	}}  
	(\text{the order of $ t(A_1 W'(\tau')) - t(A_2 W'(\tau')) $ with respect to $ q' $}).
	\end{align*}
	
	By Proposition \ref{prop:order_t_W}, this is equal to
	\[
	\sum_{ 
		\substack{
			A_i = \smat{a_i & b_i \\ 0 & d_i} 
			\in I^{\Gamma_0(M)}_{N_i, \mathrm{mat}}, \\
			\Gamma_0(M) s = \Gamma_0(M) A_i (s'), i=1, 2 
		}}
	\min_{i=1, 2} \left\{ \frac{(d_i, m' a_i + M b_i)^2}{N_i} \right\}.
	\]
	For a positive integer $ N $ coprime to $ M $ and its positive divisor $ g $,
	set
	\[
	A(N, g) :=
	\# \left\{ A = \pmat{a & b \\ 0 & d} \in I^{\Gamma_0(M)}_{N, \mathrm{mat}}
	\relmiddle{|} g = (d, m' a + M b) \right\}.
	\]
	By Proposition \ref{prop:s=A(s')}, the intersection multiplicity is
	\[
	\delta_{s, s'}(N_1) \delta_{s, s'}(N_2)
	\sum_{\substack{
			g_1 \mid N_1, g_2 \mid N_2, \\
			m' N_1 \equiv mg_1^2, m' N_2 \equiv mg_2^2 \bmod M
		}}
	A(N_1, g_1) A(N_2, g_2)
	\min \left\{ \frac{g_1^2}{N_1}, \frac{g_2^2}{N_2} \right\}.
	\]	
	Here we have
	\[
	A(N, g) 
	= \# \left\{ (a, \bar{b}, d) \relmiddle{|} 
	ad = N, \bar{b} \in \Z/d\Z, g = (d, m' a + M b) \right\}.
	\]
	Since the plus $ m' a $ map $ \Z/d\Z \to \Z/d\Z $ induces the bijection between
	\[
	\left\{ \bar{b} \in \Z/d\Z \relmiddle{|} g = (d, M b) = (b, d) \right\}
	\]
	and
	\[
	\left\{ \bar{b} \in \Z/d\Z \relmiddle{|} g = (d, m' a + M b) \right\},
	\]
	we have
	\begin{align*}
		A(N, g) 
		&= \sum_{g \mid d \mid N} \# \{ \bar{b} \in \Z/d\Z \mid g = (b, d) \}
		= \sum_{g \mid d \mid N} \# \left( \Z / (d/g) \Z \right)^\times \\
		&= \sum_{e \mid N/g} \# \left( \Z / e \Z \right)^\times
		= \frac{N}{g}.	
	\end{align*}
	Thus the intersection multiplicity is
	\begin{align*}
		&\delta_{s, s'}(N_1) \delta_{s, s'}(N_2)
		\sum_{\substack{
				g_1 \mid N_1, g_2 \mid N_2, \\
				m' N_1 \equiv mg_1^2, m' N_2 \equiv mg_2^2 \bmod M
			}}
		\frac{N_1}{g_1} \frac{N_2}{g_2}
		\min \left\{ \frac{g_1^2}{N_1}, \frac{g_2^2}{N_2} \right\} \\
		=
		&\delta_{s, s'}(N_1) \delta_{s, s'}(N_2) 
		\sum_{\substack{
				g_1 h_1 = N_1, g_2 h_2 = N_2, \\
				g_1 s \equiv h_1 s', g_2 s \equiv h_2 s' \bmod \Z
			}}  
		\min \left\{ g_1 h_2, g_2 h_1 \right\}.
	\end{align*}
	
	Unless $ M = 25 $ and $ s, s' \in \{ 1/5, 2/5, 3/5, 4/5 \} $, the condition
	$ g_1 s \equiv h_1 s', g_2 s \equiv h_2 s' \bmod \Z $ holds for any
	$ g_1, h_1, g_2 $ and $ h_2 $.	
\end{proof}

% --------------------------------------------------------------------------

\section{The class number formulas} \label{sec:class_num_formula}

% --------------------------------------------------------------------------

In this section, let $ N_1 $ and $ N_2 $ be positive integers coprime to $ M $ and suppose $N_1 N_2$ is not a square. 

The intersection number of modular correspondences on $ X_0(M) \times X_0(M) $ is calculated as follows. 

\begin{lemma} \label{lem:global_int_num}
	We have
	\[
	(T_{N_1}^{\Gamma_0(M)} \cdot T_{N_2}^{\Gamma_0(M)})
	_{X_0(M) \times X_0(M)}
	= 2 \sigma(N_1) \sigma(N_2)
	= \sum_{a_1 d_1 = N_1, a_2 d_2 = N_2}  ( a_1 d_2 + a_2 d_1 ).
	\]
\end{lemma}

\begin{proof}
	Since $ X_0(M) $ has genus zero, $ X_0(M) \times X_0(M) $ is isomorphic to $ \bbP^1 \times \bbP^1 $. 
	The intersection number of divisors on the algebraic surface $ \bbP^1 \times \bbP^1 $ only depends on its degrees. 
	The degree of the algebraic cycle $ T_{N_i}^{\Gamma_0(M)} $ is 
	$ [\mathrm{SL}_2(\Z):\Gamma_0(N_i)] $, which is the same value in the case when $ M=1 $ which is treated in \cite{Ling}. 
	Thus the intersection number on $ \bbP^1 \times \bbP^1 $ is
	\[
	(T_{N_1}^{\Gamma_0(M)} \cdot T_{N_2}^{\Gamma_0(M)})_{\bbP^1 \times \bbP^1}
	= 2 \sigma_1(N_1) \sigma_1(N_2)
	\]
	by \cite[Lemma 3.1]{Ling}. 
	The last equality follows from the definition of the divisor function $ \sigma(N) $.
\end{proof}

By combining Proposition \ref{prop:cusp_int_mult} and Lemma \ref{lem:global_int_num}, we can calculate the intersection number of $ T_{N_1}^{\Gamma_0(M)} $ and $ T_{N_2}^{\Gamma_0(M)} $ on $ X_0(M) \times X_0(M) $ as follows. 

\begin{theorem} \label{thm:main_cusp}
	Unless $ M = 25 $ and $ N_1 \equiv \pm N_2 \bmod 5 $, we have
	\[
	(T_{N_1}^{\Gamma_0(M)} \cdot T_{N_2}^{\Gamma_0(M)})
	_{Y_0(M) \times Y_0(M)}
	= 2 \sum_{a_1 d_1 = N_1, a_2 d_2 = N_2, a_1 d_2 > a_2 d_1} 
	( a_1 d_2 - \delta_M(N_1, N_2) a_2 d_1 )
	\]
	where
	\begin{align*}
	\delta_M(N_1, N_2) &:= 
	-1 + \sum_{ (s, s') \in C_0(M) } 
	\delta_{s, s'}(N_1) \delta_{s, s'}(N_2) \\
	&=
	\begin{cases}
	0 & \text{if } M = 1, \\
	1 & \text{if } M = 2, 3, 5, 7, 13, \\
	2 & \text{if } M = 4, \\
	3 & \text{if } M = 6, 8, 10, \\
	3 & \text{if } M = 9, N_1 \equiv N_2 \bmod 3, \\
	1 & \text{if } M = 9, N_1 \equiv - N_2 \bmod 3, \\
	5 & \text{if } M = 12, \\
	5 & \text{if } M = 16, N_1 \equiv N_2 \bmod 4, \\
	3 & \text{if } M = 16, N_1 \not\equiv N_2 \bmod 4, \\
	7 & \text{if } M = 18, N_1 \equiv N_2 \equiv 1 \bmod 6, \\
	5 & \text{if } M = 18, N_1 \equiv N_2 \equiv -1 \bmod 6, \\
	3 & \text{if } M = 18, N_1 \not\equiv N_2 \bmod 6, \\
	1 & \text{if } M = 25, N_1 \not\equiv \pm N_2 \bmod 5. \\
	\end{cases}
\end{align*}
	
	If $ M=25 $ and $ N_1 \equiv \pm N_2 \bmod 5 $, then
	\begin{align*}
	&(T_{N_1}^{\Gamma_0(25)} \cdot T_{N_2}^{\Gamma_0(25)})
	_{Y_0(25) \times Y_0(25)} \\
	= &\sum_{a_1 d_1 = N_1, a_2 d_2 = N_2} 
	|a_1 d_2 - a_2 d_1| 
	- 4 \sum_{\substack{
				a_1 d_1 = N_1, a_2 d_2 = N_2, \\
				a_1 d_2 \equiv a_2 d_1 \bmod 5
			}} 
	\min \left\{ a_1 d_2, a_2 d_1 \right\}.
	\end{align*}
	In particular, if $ M $ is a prime number $ p $, that is,
	$ M = p = 2, 3, 5, 7, 13 $, we have
	\[
	(T_{N_1}^{\Gamma_0(p)} \cdot T_{N_2}^{\Gamma_0(p)})
	_{Y_0(p) \times Y_0(p)}
	= \sum_{a_1 d_1 = N_1, a_2 d_2 = N_2} 
	|a_1 d_2 - a_2 d_1|.
	\]
\end{theorem}

\begin{proof}
	By Theorem \ref{thm:main_cusp_intersection}, the intersection number is
	\begin{align*}
	&(T_{N_1}^{\Gamma_0(M)} \cdot T_{N_2}^{\Gamma_0(M)})_{Y_0(M)^2} \\
	= &(T_{N_1}^{\Gamma_0(M)} \cdot T_{N_2}^{\Gamma_0(M)})_{\bbP^1 \times \bbP^1}
	- \sum_{ (s, s') \in C_0(M)^2 } 
	(T_{N_1}^{\Gamma_0(M)} \cdot T_{N_2}^{\Gamma_0(M)})_{(s, s')}.
	\end{align*}
	Unless $ M = 25 $ and $ N_1 \equiv \pm N_2 \bmod 5 $, the intersection number is
	\begin{align*}
	&\sum_{a_1 d_1 = N_1, a_2 d_2 = N_2}  ( a_1 d_2 + a_2 d_1 )
	- (1 + \delta_M(N_1, N_2)) 
	\sum_{a_1 d_1 = N_1, a_2 d_2 = N_2}  \min \left\{ a_1 d_2, a_2 d_1 \right\} \\
	= &2 \sum_{a_1 d_1 = N_1, a_2 d_2 = N_2, a_1 d_2 > a_2 d_1} 
	( a_1 d_2 - \delta_M(N_1, N_2) a_2 d_1 )
	\end{align*}
	by Proposition \ref{prop:cusp_int_mult}. 
	If $ M = 25 $ and $ N_1 \equiv \pm N_2 \bmod 5 $, then the intersection number is
	\begin{align*}
	&(T_{N_1}^{\Gamma_0(25)} \cdot T_{N_2}^{\Gamma_0(25)})_{\bbP^1 \times \bbP^1}
	- (T_{N_1}^{\Gamma_0(25)} \cdot T_{N_2}^{\Gamma_0(25)})_{(i\infty, i\infty)}
	- (T_{N_1}^{\Gamma_0(25)} \cdot T_{N_2}^{\Gamma_0(25)})_{(0, 0)} \\
	&- \sum_{s, s' \in \{ 1/5, 2/5, 3/5, 4/5 \}} 
	(T_{N_1}^{\Gamma_0(25)} \cdot T_{N_2}^{\Gamma_0(25)})_{(s, s')} \\
	= &\sum_{a_1 d_1 = N_1, a_2 d_2 = N_2} 
	|a_1 d_2 - a_2 d_1| 
	- \sum_{s, s' \in \{ 1/5, 2/5, 3/5, 4/5 \}} 
	\sum_{\substack{
			a_1 d_1 = N_1, a_2 d_2 = N_2, \\
			a_1 s \equiv d_1 s', a_2 s \equiv d_2 s' \bmod \Z
	}}  
	\min \left\{ a_1 d_2, a_2 d_1 \right\} \\
	= &\sum_{a_1 d_1 = N_1, a_2 d_2 = N_2} 
	|a_1 d_2 - a_2 d_1| 
	- 4\sum_{\substack{
			a_1 d_1 = N_1, a_2 d_2 = N_2, \\
			a_1 d_2 \equiv a_2 d_1 \bmod 5
	}}  
	\min \left\{ a_1 d_2, a_2 d_1 \right\}.
	\end{align*}
\end{proof}

We have the following main Theorem in this paper. 

\begin{theorem}
	Unless $ M = 25 $ and $ N_1 \equiv \pm N_2 \bmod 5 $, we have
	\begin{align*}
	&\sum_{x \in \Z,\ x^2 < 4N_1 N_2} \sum_{d \mid (N_1, N_2, x)}
	d \cdot H^M \left( \frac{4N_1 N_2- x^2}{d^2} \right) \\
	= 2 &\sum_{a_1 d_1 = N_1, a_2 d_2 = N_2, a_1 d_2 > a_2 d_1} 
	( a_1 d_2 - \delta_M(N_1, N_2) a_2 d_1 )
	\end{align*}
	where $ \delta_M(N_1, N_2) $ is defined in Theorem $ \ref{thm:main_cusp} $. 
	If $ M = 25 $ and $ N_1 \equiv \pm N_2 \bmod 5 $, then
	\begin{align*}
	&\sum_{x \in \Z,\ x^2 < 4N_1 N_2} \sum_{d \mid (N_1, N_2, x)}
	d \cdot H^{25} \left( \frac{4N_1 N_2- x^2}{d^2} \right) \\
	= &\sum_{a_1 d_1 = N_1, a_2 d_2 = N_2} 
	|a_1 d_2 - a_2 d_1| 
	- 8 \sum_{\substack{
			a_1 d_1 = N_1, a_2 d_2 = N_2, a_1 d_2 > a_2 d_1, \\
			a_1 d_2 \equiv a_2 d_1 \bmod 5
	}}  
	a_2 d_1.
	\end{align*}
\end{theorem}

\begin{proof}
	It follows from Theorem \ref{thm:main_Y_0(M)} and Theorem \ref{thm:main_cusp}.
\end{proof}

Theorem \ref{thm:main_class_num} is the special case in this theorem.

% --------------------------------------------------------------------------

\section{Explicit computation of $ H^M(D) $ for some $ M $} \label{sec:computation}

% --------------------------------------------------------------------------

In this section, we give a method to compute the Hurwitz class number $ H^M(D) $ for $ M $ when $ 2 \le M \le 10 $ or $ M \in \{ 12, 13, 16, 18, 25 \} $.

If the level is a prime number $ p \in \{ 2, 3, 5, 7, 13 \} $, then $ H^p(D) $ can be calculated from $ H(D) $ by the following theorem.

\begin{theorem}[{\cite[Lemma 3.2]{CK}}] \label{thm:Choi-Kim}
	For a prime number $ p \in \{ 2, 3, 5, 7, 13 \} $ and a positive integer $ D \equiv 0, 3 \bmod 4 $, we have
	\[
	H^p(D)
	= \left( 1 + \left( \frac{-D}{p} \right) \right)
	\left( H(D) + p \cdot H \left( \frac{D}{p^2} \right) \right).
	\]
	Here we define $ H(D/p^2) := 0 $ if $ p^2 \nmid D $.
\end{theorem}

This theorem is slightly different from the original statement of \cite[Lemma 3.2]{CK}.
See also a proof of \cite[Proposition 3.3]{Mura}.

In general, we can calculate the Hurwitz class number $ H^M(D) $ by considering a fundamental domain of $ \Gamma_0(M) $.
The following elementary lemma is useful for computing $ H^M(D) $.

\begin{lemma} \label{lem:w_Q_inequality}
	Let $ D \equiv 0, 3 \bmod 4 $ be a positive integer, $ [Ma, b, c] \in \calQ_{-D, >0}^{M} $ with $ a, c \ge 1 $ and
	\[
	w_Q^{} := \frac{-b + \sqrt{-D}}{2Ma}.
	\]	
	For positive integers $ m $ and $ n $, if
	\[
	\abs{w_Q^{} \pm \frac{m}{n}} \ge \frac{1}{n},
	\]
	then $ \pm b \le ka + lc $ where
	\[
	k := \frac{M(m^2 - 1)}{mn}, \quad
	l := \frac{n}{m}.
	\]
	Moreover, if $ m \ge 2 $, $ \abs{b} \le ka + lc $, and
	$ r_{-}c + t_{-} \le a \le r_{+}c - t_{+} $ with $ t_{+}, t_{-} > 0 $ and
	\[
	r_{+} := \frac{1}{M} \left( \frac{n}{m - 1} \right)^2, \quad
	r_{-} := \frac{1}{M} \left( \frac{n}{m + 1} \right)^2,
	\]
	then $ a \le D/C $ where
	\[
	C := \min \left\{ 
	t_{+} \left( k^2 - \frac{l^2}{r_{+}(r_{+} - t_{+})} \right), \quad
	-t_{-} \left( k^2 - \frac{l^2}{r_{-}(r_{-} + t_{-})} \right)
	\right\}.
	\]
\end{lemma}

\begin{proof}
	By direct calculation, the condition 
	\[
	\abs{w_Q^{} \pm \frac{m}{n}} \ge \frac{1}{n}
	\]
	is equivalent to $ Mam^2 \mp bmn + cn^2 \ge Ma $, that is, $ \pm b \le ka + lc $.
	
	Suppose $ m \ge 2 $ and $ \abs{b} \le ka + lc $.
	Let $ f(x) := k^2 x + l^2/x $.
	Then we have
	\[
	D \ge 4Mac - (ka + lc)^2
	= ac \left(4M - 2kl - f\left( \frac{a}{c} \right)\right).
	\]
	By direct calculation, $ f(x) = 4M-2kl $ if and only if 
	$ x = r_{+} $ or $ x = r_{-} $.
	Moreover, if $ r_{-} \le x \le l/k $ then $ f(x) $ is monotonic decreasing and if $ l/k \le x \le r_{+} $ then $ f(x) $ is monotonic increasing by elementary calculus.
	For $ t_{+}, t_{-} > 0 $, we have
	\[
	f\left( r_{-} + \frac{t_{-}}{c} \right)
	= f(r_{-}) + \frac{t_{-}}{c} \left( k^2 - \frac{l^2}{r_{-} (r_{-} + t_{-}/c)} \right)
	\le f(r_{-}) + \frac{t_{-}}{c} \left( k^2 - \frac{l^2}{r_{-} (r_{-} + t_{-})} \right)
	\]
	and
	\[
	f\left( r_{+} - \frac{t_{+}}{c} \right)
	= f(r_{+}) - \frac{t_{+}}{c} \left( k^2 - \frac{l^2}{r_{+} (r_{+} - t_{+}/c)} \right)
	\le f(r_{+}) + \frac{t_{+}}{c} \left( k^2 - \frac{l^2}{r_{+} (r_{+} - t_{+})} \right).
	\]
	Hence if $ r_{-}c + t_{-} \le a \le r_{+}c - t_{+} $, then
	\[
	D \ge ac \left(4M - 2kl - f\left( \frac{a}{c} \right)\right)
	\ge aC.
	\]
\end{proof}

Here we calculate $ H^M(D) $ when $ M $ is a composite number by the following lemma.

\begin{lemma}
	Let $ D \equiv 0, 3 \bmod 4 $ be a positive integer.
	\begin{enumerate}
		\item \label{item:lem:simple_form_M=4}
		The set
		\[
		\left\{ [4a, b, c] \in \calQ_{-D, >0}^{4} \relmiddle|
		\begin{gathered}
			\abs{b} \le 4 \min\{ a, c \}, \\
			\text{if } \abs{b} = 4 \min\{ a, c \} \text{ then } 0 \le b		
		\end{gathered}
		\right\}
		\]
		is a complete system of representatives of $ \calQ_{-D, >0}^{4}/\Gamma_0(4) $.
		Moreover, if $ [4a, b, c] $ is an element of this set, then 
		$ a, c \le (D+1)/8 $.
		\item \label{item:lem:simple_form_M=6}
		The set
		\[
		\left\{ [6a, b, c] \in \calQ_{-D, >0}^{6} \relmiddle|
		\begin{gathered}
			\abs{b} \le 6 \min\{ a, c, (2/5)(a + c) \}, \\
			\text{if } \abs{b} = 6 \min\{ a, c, (2/5)(a + c)  \} \text{ then } 0 \le b		
		\end{gathered}
		\right\}
		\]
		is a complete system of representatives of $ \calQ_{-D, >0}^{6}/\Gamma_0(6) $.
		Moreover, if $ [6a, b, c] $ is an element of this set, then 
		$ a, c \le (25/24)D $.
		\item \label{item:lem:simple_form_M=8}
		The set
		\[
		\left\{ [8a, b, c] \in \calQ_{-D, >0}^{8} \relmiddle|
		\begin{gathered}
			\abs{b} \le 8a, -8c \le b \le 4c, -(8/7)(2a + 3c) \le b, -(4/5)(4a + 3c) \le b \\
			\text{if } 
			b \in \{  
			\pm 8a, -8c, 4c, -(8/7)(2a + 3c), -(4/5)(4a + 3c)
			\} \\
			\text{ then }
			-4a \le b		
		\end{gathered}
		\right\}
		\]
		is a complete system of representatives of $ \calQ_{-D, >0}^{8}/\Gamma_0(8) $.
		Moreover, if $ [8a, b, c] $ is an element of this set, then 
		$ a, c \le (245/96)D $.
		\item \label{item:lem:simple_form_M=9}
		The set
		\[
		\left\{ [9a, b, c] \in \calQ_{-D, >0}^{9} \relmiddle|
		\begin{gathered}
			\abs{b} \le 9 \min\{ a, c, (2/5)(3a + 2c) \}, \\
			\text{if } \abs{b} = 9 \min\{ a, c, (2/5)(3a + 2c) \} 
			\text{ then } 0 \le b		
		\end{gathered}
		\right\}
		\]
		is a complete system of representatives of $ \calQ_{-D, >0}^{9}/\Gamma_0(9) $.
		Moreover, if $ [9a, b, c] $ is an element of this set, then 
		$ a, c \le (25/72)D $.
		\item \label{item:lem:simple_form_M=10}
		The set
		\[
		\left\{ [10a, b, c] \in \calQ_{-D, >0}^{10} \relmiddle|
		\begin{gathered}
			\abs{b} \le 10 \min\{ a, (3/5)c, (1/11)(4a + 3c) , (2/9)(2a + c) \}, \\
			\text{if } \abs{b} = 10 \min\{ a, (3/5)c, (1/11)(4a + 3c) , (2/9)(2a + c)  \} \\
			\text{ then } (20/3)a \le b \text{ or } \abs{b} \le 6a		
		\end{gathered}
		\right\}
		\]
		is a complete system of representatives of $ \calQ_{-D, >0}^{10}/\Gamma_0(10) $.
		Moreover, if $ [10a, b, c] $ is an element of this set, then 
		$ a, c \le (121/35)D $.
		\item \label{item:lem:simple_form_M=12}
		The set
		\[
		\left\{ [12a, b, c] \in \calQ_{-D, >0}^{12} \relmiddle|
		\begin{gathered}
			\abs{b} \le \min\{ 12a, (12/5)(2a + c),(24/7)(a + c) \}, \\
			-12c \le b \le 8c,
			b \ge \max \{ -(12/11)(2a + 5c), -(8/9)(3a + 5c) \}, \\
			\text{if }
			\abs{b} = \min\{ 12a, (12/5)(2a + c),(24/7)(a + c) \},
			b = -12c, b = 8c
			\text{ or } \\
			b = \max \{ -(12/11)(2a + 5c), -(8/9)(3a + 5c) \},
			\text{ then }
			-4a \le b \le 12a		
		\end{gathered}
		\right\}
		\]
		is a complete system of representatives of $ \calQ_{-D, >0}^{12}/\Gamma_0(12) $.
		Moreover, if $ [12a, b, c] $ is an element of this set, then 
		$ a, c \le (1573/240)D $.
		\item \label{item:lem:simple_form_M=16}
		The set
		\[
		\left\{ [16a, b, c] \in \calQ_{-D, >0}^{16} \relmiddle|
		\begin{gathered}
			\abs{b} \le \min\{ 16a, 8c, (8/7)(4a + 3c) \},
			b \le (4/5)(8a + 3c), \\
			b \ge \max \{ -(48/17)(2a + c), -(16/31)(12a + 5c), -(4/9)(16a + 5c) \}, \\
			\text{if }
			\abs{b} = \min\{ 16a, 8c, (8/7)(4a + 3c) \},
			b = (4/5)(8a + 3c)
			\text{ or } \\
			b = \max \{ -(48/17)(2a + c), -(16/31)(12a + 5c) \}, \\
			\text{ then }
			\abs{b} \le	8a, b \ge (32/3)a
			\text{ or }	
			-12a \le b \le -(32/3)a
		\end{gathered}
		\right\}
		\]
		is a complete system of representatives of $ \calQ_{-D, >0}^{16}/\Gamma_0(16) $.
		Moreover, if $ [16a, b, c] $ is an element of this set, then 
		$ a, c \le D $.
		\item \label{item:lem:simple_form_M=18}
		The set
		\[
		\left\{ [18a, b, c] \in \calQ_{-D, >0}^{18} \relmiddle|
		\begin{gathered}
			\abs{b} \le \min \left\{
			\begin{gathered}
				18a, 12c,
				(12/11)(3a + 5c),
				(18/19)(4a + 5c), \\
				(12/7)(3a + 2c),
				(12/5)(3a + c),
				(72/17)(a + c)
			\end{gathered}
			\right\}, \\
			\text{if }
			\abs{b} = \min \left\{
			\begin{gathered}
			18a, 12c,
			(12/11)(3a + 5c),
			(18/19)(4a + 5c), \\
			(12/7)(3a + 2c),
			(12/5)(3a + c),
			(72/17)(a + c)
			\end{gathered}
			\right\}, \\
			\text{ then }
			\abs{b} \le	6a,
			(36/5)a \le b \le 9a,
			12a \le \abs{b} < 18a	
			\text{ or }
			b = 18a
		\end{gathered}
		\right\}
		\]
		is a complete system of representatives of $ \calQ_{-D, >0}^{18}/\Gamma_0(18) $.
		Moreover, if $ [18a, b, c] $ is an element of this set, then 
		$ a, c \le (361/45)D $.
		\item \label{item:lem:simple_form_M=25}
		The set
		\[
		\left\{ [25a, b, c] \in \calQ_{-D, >0}^{25} \relmiddle|
		\begin{gathered}
			\abs{b} \le \min \left\{
			\begin{gathered}
			25a, 10c,
			(10/9)(5a + 4c),
			(2/7)(25a + 12c), \\
			(10/11)(10a + 3c),
			(20/9)(5a + c)
			\end{gathered}
			\right\}, \\
			\text{if }
			\abs{b} = \min \left\{
			\begin{gathered}
			25a, 10c,
			(10/9)(5a + 4c),
			(2/7)(25a + 12c), \\
			(10/11)(10a + 3c),
			(20/9)(5a + c)
			\end{gathered}
			\right\}, \\
			\text{ then }
			\abs{b} \le	10a,
			14a \le \abs{b} \le 20a
			\text{ or }
			b = 25a
		\end{gathered}
		\right\}
		\]
		is a complete system of representatives of $ \calQ_{-D, >0}^{25}/\Gamma_0(25) $.
		Moreover, if $ [25a, b, c] $ is an element of this set, then 
		$ a, c \le (968/175)D $.
	\end{enumerate}
\end{lemma}

\begin{proof}
	Firstly, we get a complete system of representatives of $ \calQ_{-D, >0}^{M}/\Gamma_0(M) $ in each case by Lemma \ref{lem:w_Q_inequality} since we have fundamental domains of $ \Gamma_0(4), \Gamma_0(6), \Gamma_0(8), \Gamma_0(9), \Gamma_0(10), \Gamma_0(12) $,  $ \Gamma_0(16) $, $ \Gamma_0(18) $, and $ \Gamma_0(25) $ as
	\[
	\left\{ \tau \in \bbH \relmiddle|
	\begin{gathered}
	\abs{\ReNew(\tau)} \le 1/2, 
	\abs{\tau \pm 1/4} \ge 1/4, \\
	\text{if } \abs{\ReNew(\tau)} = 1/2
	\text{ or } \abs{\tau \pm 1/4} = 1/4
	\text{ then } \ReNew(\tau) \le 0
	\end{gathered}
	\right\},
	\]
	\[
	\left\{ \tau \in \bbH \relmiddle|
	\begin{gathered}
	\abs{\ReNew(\tau)} \le 1/2, 
	\abs{\tau \pm 1/6} \ge 1/6,
	\abs{\tau \pm 5/12} \ge 5/12, \\
	\text{if } \abs{\ReNew(\tau)} = 1/2,
	\abs{\tau \pm 1/6} = 1/6
	\text{ or } \abs{\tau \pm 5/12} \ge 5/12
	\text{ then } \ReNew(\tau) \le 0
	\end{gathered}
	\right\},
	\]
	\[
	\left\{ \tau \in \bbH \relmiddle|
	\begin{gathered}
	\abs{\ReNew(\tau)} \le 1/2, 
	\abs{\tau + 1/4} \ge 1/4,
	\abs{\tau - 1/8} \ge 1/8, \\
	\abs{\tau - 7/24} \ge 1/24, 
	\abs{\tau - 5/12} \ge 1/12, \\
	\text{if } \abs{\ReNew(\tau)} = 1/2,
	\abs{\tau + 1/4} = 1/4,
	\abs{\tau - 1/8} = 1/8,
	\abs{\tau - 7/24} = 1/24 \\
	\text{ or } \abs{\tau - 5/12} = 1/12	
	\text{ then } \ReNew(\tau) \le 1/4
	\end{gathered}
	\right\},
	\]
	\[
	\left\{ \tau \in \bbH \relmiddle|
	\begin{gathered}
	\abs{\ReNew(\tau)} \le 1/2, 
	\abs{\tau \pm 1/6} \ge 1/6,
	\abs{\tau \pm 5/12} \ge 5/12, \\
	\text{if } \abs{\ReNew(\tau)} = 1/2,
	\abs{\tau \pm 1/6} = 1/6
	\text{ or } \abs{\tau \pm 5/12} \ge 5/12
	\text{ then } \ReNew(\tau) \le 0
	\end{gathered}
	\right\},
	\]
	\begin{equation} \label{eq:fund_dom_M=10}
		\left\{ \tau \in \bbH \relmiddle|
		\begin{gathered}
		\abs{\ReNew(\tau)} \le 1/2, 
		\abs{\tau \pm 1/6} \ge 1/6,
		\abs{\tau \pm 11/30} \ge 1/30, 
		\abs{\tau \pm 9/20} \ge 1/20, \\
		\text{if } \abs{\ReNew(\tau)} = 1/2,
		\abs{\tau \pm 1/6} = 1/6,
		\abs{\tau \pm 11/30} = 1/30
		\text{ or } \abs{\tau \pm 9/20} = 1/20 \\
		\text{ then } \ReNew(\tau) \le -1/3
		\text{ or } \abs{\ReNew(\tau)} \le 3/10
		\end{gathered}
		\right\},	
	\end{equation}
	\[
	\left\{ \tau \in \bbH \relmiddle|
	\begin{gathered}
	\abs{\ReNew(\tau)} \le 1/2, 
	\abs{\tau \pm 5/12} \ge 1/12,
	\abs{\tau \pm 7/24} \ge 1/24,
	\abs{\tau + 1/8} \ge 1/8, 
	\\
	\abs{\tau - 1/12} \ge 1/12, 
	\abs{\tau - 11/60} \ge 1/60,
	\abs{\tau - 9/40} \ge 1/40, 
	\\
	\text{if }
	\abs{\ReNew(\tau)} = 1/2,
	\abs{\tau \pm 5/12} = 1/12,
	\abs{\tau \pm 7/24} = 1/24,
	\abs{\tau + 1/8} = 1/8, 
	\\
	\abs{\tau - 1/12} = 1/12, 
	\abs{\tau - 11/60} = 1/60
	\text{ or } 
	\abs{\tau - 9/40} = 1/40 
	\\
	\text{ then } -1/2 \le \ReNew(\tau) \le 1/6
	\end{gathered}
	\right\},
	\]
	\[
	\left\{ \tau \in \bbH \relmiddle|
	\begin{gathered}
		\abs{\ReNew(\tau)} \le 1/2, 
		\abs{\tau \pm 1/8} \ge 1/8,
		\abs{\tau \pm 7/24} \ge 1/24,
		\abs{\tau + 5/12} \ge 1/12, 
		\\
		\abs{\tau - 17/48} \ge 1/48, 
		\abs{\tau - 31/80} \ge 1/80,
		\abs{\tau - 9/20} \ge 1/20,
		\\
		\text{if }
		\abs{\ReNew(\tau)} = 1/2, 
		\abs{\tau \pm 1/8} = 1/8,
		\abs{\tau \pm 7/12} = 1/24,
		\abs{\tau + 5/12} = 1/12, 
		\\
		\abs{\tau - 17/48} = 1/48, 
		\abs{\tau - 31/80} = 1/80
		\text{ or } 
		\abs{\tau - 9/20} = 1/20,
		\\
		\text{ then } 
		\abs{\ReNew(\tau)} \le 1/4,
		\ReNew(\tau) \le -1/3
		\text{ or }
		1/3 \le \ReNew(\tau) \le 3/8
	\end{gathered}
	\right\},
	\]
	\[
	\left\{ \tau \in \bbH \relmiddle|
	\begin{gathered}
		\abs{\ReNew(\tau)} \le 1/2, 
		\abs{\tau \pm 1/12} \ge 1/12,
		\abs{\tau \pm 11/60} \ge 1/60,
		\\
		\abs{\tau \pm 19/90} \ge 1/90,
		\abs{\tau \pm 17/72} \ge 1/72,
		\abs{\tau \pm 7/24} \ge 1/24,
		\abs{\tau \pm 5/12} \ge 1/12,
		\\
		\text{if }
		\abs{\ReNew(\tau)} = 1/2,
		\abs{\tau \pm 1/12} = 1/12,
		\abs{\tau \pm 11/60} = 1/60,
		\\
		\abs{\tau \pm 19/90} = 1/90,
		\abs{\tau \pm 17/72} = 1/72,
		\abs{\tau \pm 7/24} = 1/24
		\text{ or } 
		\abs{\tau \pm 5/12} = 1/12,
		\\
		\text{ then }
		\abs{\ReNew(\tau)} \le 1/6,
		-1/4 \le \ReNew(\tau) \le -1/5,
		1/3 \le \abs{\ReNew(\tau)} < 1/2 
		\text{ or }
		\abs{\ReNew(\tau)} = -1/2 
	\end{gathered}
	\right\},
	\]
	and
	\begin{equation} \label{eq:fund_dom_M=25}
		\left\{ \tau \in \bbH \relmiddle|
		\begin{gathered}
		\abs{\ReNew(\tau)} \le 1/2, 
		\abs{\tau \pm 1/10} \ge 1/10,
		\abs{\tau \pm 9/40} \ge 1/40,
		\\
		\abs{\tau \pm 7/24} \ge 1/24,
		\abs{\tau \pm 11/30} \ge 1/30,
		\abs{\tau \pm 9/20} \ge 1/20,
		\\
		\text{if }
		\abs{\ReNew(\tau)} = 1/2,
		\abs{\tau \pm 1/10} = 1/10,
		\abs{\tau \pm 9/40} = 1/40,
		\\
		\abs{\tau \pm 7/24} = 1/24,
		\abs{\tau \pm 11/30} = 1/30
		\text{ or } 
		\abs{\tau \pm 9/20} = 1/20,
		\\
		\text{ then }
		\abs{\ReNew(\tau)} \le 1/5,
		7/25 \le \abs{\ReNew(\tau)} \le 2/5	
		\text{ or }
		\abs{\ReNew(\tau)} = -1/2 
		\end{gathered}
		\right\}
	\end{equation}
	by using the algorithm in \cite{Kurth-Long} which is based on the theory of Farey symbols in \cite{Kulkarni} and is implemented for Sage \cite{sagemath} by Chris A. Kurth.
	
	Secondly, we bound $ a $ and $ c $ in each case.
	
	\ref{item:lem:simple_form_M=4}.
	For a quadratic form $ [4a, b, c] $, $ (-b + \sqrt{b^2 - 4ac})/2 $ is a point of above fundamental domain if and only if $ [4a, b, c] $ is an element of the set in statement.
	In this case, if $ a > c $ then 
	\[
	D = -b^2 + 16ac \ge -16c^2 + 16ac = 16c(a-c)
	\]
	and thus we have $ c \le D/16 $ and $ a-c \le D/16 $.
	Then we obtain $ c < a \le D/8 $.
	Similarly, if $ a < c $ then we have $ a < c \le D/8 $.
	If $ a=c $, then $ \abs{b} \le 4a-1 $ since $ 0 < D = -b^2 + 16a^2 $.
	Thus we have $ D \ge -(4a-1)^2 + 16a^2 = 8a-1 $.
		
	For other cases, we have the boundings by Lemma \ref{lem:w_Q_inequality} and similar argument in \ref{item:lem:simple_form_M=4}.
	For \ref{item:lem:simple_form_M=6}, we have:
	\begin{quote}
		\begin{enumerate}[(a)]
			\item If $ 3a < 2c $ or $ 3c < 2a $, then $ a, c \le D/6 $.
			\item If $ 3a = 2c $ or $ 3c = 2a $, then $ a, c \le (D+1)/8 $.
			\item If $ 2a \le 2c < 3a $ or $ 2c \le 2a < 3c $, then $ a, c \le (25/24)D $.
		\end{enumerate}		
	\end{quote}
	For \ref{item:lem:simple_form_M=8}, we have:
	\begin{quote}
		\begin{enumerate}[(a)]
			\item If $ 2a < c $ or $ 2c < a $, then $ a, c \le (3/16)D $.
			\item If $ 2a = c $ or $ 2c = a $, then $ a, c \le (D+1)/8 $.
			\item If $ c < 2a $ and $ b \ge 0 $, then $ a, c \le D/16 $.
			\item If $ 4c < 8a < 9c $ and $ b \le 0 $, then $ a, c \le (25/16)D $.
			\item If $ 9c < 8a < 16c $ and $ b \le 0 $, then $ a, c \le (245/96)D $.
		\end{enumerate}		
	\end{quote}
	For \ref{item:lem:simple_form_M=9}, we have:
	\begin{quote}
		\begin{enumerate}[(a)]
			\item If $ 9a < 4c $, then $ a, c \le (5/18)D $.
			\item If $ 9a = 4c $, then $ a, c \le (D+1)/4 $.
			\item If $ 4a < 4c < 9a $, then $ a, c \le (25/72)D $.
			\item If $ a = c $, then $ a, c \le (D+1)/6 $.
			\item If $ c < a $, then $ a, c \le D/18 $.
		\end{enumerate}
	\end{quote}
	For \ref{item:lem:simple_form_M=10}, we have:
	\begin{quote}
		\begin{enumerate}[(a)]
			\item If $ 5a < 2c $, then $ a, c \le (3/20)D $.
			\item If $ 5a = 2c $, then $ a, c \le (D+1)/8 $.
			\item If $ (2/5)c < a < (5/8)c $, then $ a, c \le (81/40)D $.
			\item If $ 8a = 5c $, then $ a, c \le (D+1)/10 $.
			\item If $ (5/8)c < a < (9/10)c $ , then $ a, c \le (121/35)D $.
			\item If $ 10a = 9c $, then $ a, c \le (D+1)/12 $.
			\item If $ 9c < 10a $, then $ a, c \le D/4 $.
		\end{enumerate}	
	\end{quote}
	For \ref{item:lem:simple_form_M=12}, we have:
	\begin{quote}
		\begin{enumerate}[(a)]
			\item If $ 3a < c $ or $ 3c < a $, then $ a, c \le D/12 $.
			\item If $ 3a = c $ or $ 3c = a $, then $ a, c \le (D+1)/8 $.
			\item If $ (1/3)c < a < (3/4)c $, then $ a, c \le (25/8)D $.
			\item If $ a = (3/4)c $ or $ a = (4/3) $, then $ a, c \le (D+1)/12 $.
			\item If $ (3/4)c < a < (4/3)c $, then $ a, c \le (49/36)D $.
			\item If $ (4/3)c < a < (25/12)c $, then $ a, c \le (243/64)D $.
			\item If $ a = (25/12)c $, then $ a, c \le (3/125)(D+1) $.
			\item If $ (25/12)c < a < 3c $, then $ a, c \le (1573/240)D $.
		\end{enumerate}	
	\end{quote}	
	For \ref{item:lem:simple_form_M=16}, we have:	
	\begin{quote}
		\begin{enumerate}[(a)]
			\item If $ 4a < c $, then $ a, c \le (5/64)D $.
			\item If $ 4a = c $, then $ a, c \le (D+1)/8 $.
			\item If $ (1/4)c < a < (9/16)c $, then $ a, c \le (25/32)D $.
			\item If $ 16a = 9c $, then $ a, c \le (D+1)/12 $.
			\item If $ (9/16)c < a < c $, then $ a, c \le (245/192)D $.
			\item If $ a = c $, then $ a, c \le (D+1)/32 $.
			\item If $ c < a $, then $ a, c \le D/32 $.
		\end{enumerate}	
	\end{quote}
	For \ref{item:lem:simple_form_M=18}, we have:	
	\begin{quote}
		\begin{enumerate}[(a)]
			\item If $ 9a < 2c $, then $ a, c \le (5/36)D $.
			\item If $ 9a = 2c $, then $ a, c \le (D+1)/8 $.
			\item If $ (2/9)c < a < (1/2)c $, then $ a, c \le (425/72)D $.
			\item If $ 2a = c $ or $ a = 2c $, then $ a, c \le (D+1)/24 $.
			\item If $ (1/2)c < a < (8/9)c $, then $ a, c \le (686/360)D $.
			\item If $ 9a = 8c $ or $ 8a = 9c $, then $ a, c \le (D+1)/16 $.
			\item If $ (8/9)c < a < (9/8)c $, then $ a, c \le (289/64)D $.
			\item If $ (9/8)c < a < (25/18)c $, then $ a, c \le (361/45)D $.
			\item If $ 18a = 25c $, then $ a, c \le (5/72)(D+1) $.
			\item If $ (25/18)c < a < 2c $, then $ a, c \le (1573/360)D $.
			\item If $ 2c < a $, then $ a, c \le D/24 $.
		\end{enumerate}	
	\end{quote}
	For \ref{item:lem:simple_form_M=25}, we have:	
	\begin{quote}
		\begin{enumerate}[(a)]
			\item If $ 25a < 4c $, then $ a, c \le (3/50)D $.
			\item If $ 25a = 4c $, then $ a, c \le (D+1)/80 $.
			\item If $ (4/25)c < a < (1/4)c $, then $ a, c \le (81/16)D $.
			\item If $ 4a = c $, then $ a, c \le (D+1)/100 $.
			\item If $ (1/4)c < a < (9/25)c $, then $ a, c \le (968/175)D $.
			\item If $ 25a = 9c $, then $ a, c \le (D+1)/120 $.
			\item If $ 9c < 25a < 16c $, then $ a, c \le (245/72)D $.
			\item If $ 25a = 16c $, then $ a, c \le (D+1)/160 $.
			\item If $ c < a $, then $ a, c \le D/100 $.
		\end{enumerate}	
	\end{quote}
\end{proof}

To compute $ H^M(D) $, we need a criterion whether the stabilizer $ \Gamma_0(M)_Q $ is non-trivial for a quadratic form $ Q \in \calQ_{-D, >0}^{M} $.
Such quadratic forms correspond to elliptic points for $ \Gamma_0(M) $.
By \cite[Corollary 3.7.2]{DS}, $ \Gamma_0(M) $ has no elliptic points for $ M \in \{ 4, 6, 8, 9, 12, 16, 18 \} $ and has exactly 2 elliptic points of period 2 and no elliptic points of period 3 for $ M \in \{ 10, 25 \} $.
For $ M = 10 $, elliptic points in the fundamental domain in (\ref{eq:fund_dom_M=10}) are $ (\pm 3 + \sqrt{-1})/10 $ whose corresponding quadratic forms are $ [10, \mp 6, 1] $.
For $ M = 25 $, elliptic points in the fundamental domain in (\ref{eq:fund_dom_M=25}) are $ (\pm 7 + \sqrt{-1})/25 $ whose corresponding quadratic forms are $ [25, \mp 14, 2] $.

We show $ H(D) $ and $ H^M(D) $ for a positive integer $ D \le 50 $ in Tables \ref{tab:Hurwitz_class_num1} and \ref{tab:Hurwitz_class_num2}.

\begin{table}[htb]
	\caption{$ H(D) $ and $ H^M(D) $ for positive integers $ D \le 100 $.}
	\label{tab:Hurwitz_class_num1}
	\centering
	\begin{tabular}{ccccccccc}
		\hline\noalign{\smallskip}
		$ D $ & $ H(D) $ & $ H^2(D) $ & $ H^3(D) $ & $ H^4(D) $ & $ H^5(D) $ & $ H^6(D) $ & $ H^7(D) $ & $ H^{8}(D) $ \\
		\hline
		\rowcolor[gray]{0.95}
		0 & $ -1/12 $ & $ -1/4 $ & $ -1/3 $ & $ -1/2 $ & $ -1/2 $ & $ -1 $ & $ -2/3 $ & $ -1 $ \\ 
		3 & 1/3 & 0 & 1/3 & 0 & 0 & 0 & 2/3 & 0 \\ 
		\rowcolor[gray]{0.95}
		4 & 1/2 & 1/2 & 0 & 0 & 1 & 0 & 0 & 0 \\ 
		7 & 1 & 2 & 0 & 2 & 0 & 0 & 1 & 2 \\ 
		\rowcolor[gray]{0.95}
		8 & 1 & 1 & 2 & 0 & 0 & 2 & 0 & 0 \\ 
		11 & 1 & 0 & 2 & 0 & 2 & 0 & 0 & 0 \\ 
		\rowcolor[gray]{0.95}
		12 & 4/3 & 2 & 4/3 & 2 & 0 & 2 & 8/3 & 0 \\ 
		15 & 2 & 4 & 2 & 4 & 2 & 4 & 0 & 4 \\ 
		\rowcolor[gray]{0.95}
		16 & 3/2 & 5/2 & 0 & 3 & 3 & 0 & 0 & 2 \\ 
		19 & 1 & 0 & 0 & 0 & 2 & 0 & 2 & 0 \\ 
		\rowcolor[gray]{0.95}
		20 & 2 & 2 & 4 & 0 & 2 & 4 & 4 & 0 \\ 
		23 & 3 & 6 & 6 & 6 & 0 & 12 & 0 & 6 \\ 
		\rowcolor[gray]{0.95}
		24 & 2 & 2 & 2 & 0 & 4 & 2 & 4 & 0 \\ 
		27 & 4/3 & 0 & 7/3 & 0 & 0 & 0 & 8/3 & 0 \\ 
		\rowcolor[gray]{0.95}
		28 & 2 & 4 & 0 & 6 & 0 & 0 & 2 & 8 \\ 
		31 & 3 & 6 & 0 & 6 & 6 & 0 & 6 & 6 \\ 
		\rowcolor[gray]{0.95}
		32 & 3 & 5 & 6 & 6 & 0 & 10 & 0 & 4 \\ 
		35 & 2 & 0 & 4 & 0 & 2 & 0 & 2 & 0 \\ 
		\rowcolor[gray]{0.95}
		36 & 5/2 & 5/2 & 4 & 0 & 5 & 4 & 0 & 0 \\ 
		39 & 4 & 8 & 4 & 8 & 8 & 8 & 0 & 8 \\ 
		\rowcolor[gray]{0.95}
		40 & 2 & 2 & 0 & 0 & 2 & 0 & 4 & 0 \\ 
		43 & 1 & 0 & 0 & 0 & 0 & 0 & 0 & 0 \\ 
		\rowcolor[gray]{0.95}
		44 & 4 & 6 & 8 & 6 & 8 & 12 & 0 & 0 \\ 
		47 & 5 & 10 & 10 & 10 & 0 & 20 & 10 & 10 \\ 
		\rowcolor[gray]{0.95}
		48 & 10/3 & 6 & 10/3 & 7 & 0 & 6 & 20/3 & 8 \\ 
		51 & 2 & 0 & 2 & 0 & 4 & 0 & 0 & 0 \\ 
		\rowcolor[gray]{0.95}
		52 & 2 & 2 & 0 & 0 & 0 & 0 & 4 & 0 \\ 
		55 & 4 & 8 & 0 & 8 & 4 & 0 & 8 & 8 \\ 
		\rowcolor[gray]{0.95}
		56 & 4 & 4 & 8 & 0 & 8 & 8 & 4 & 0 \\ 
		59 & 3 & 0 & 6 & 0 & 6 & 0 & 6 & 0 \\ 
		\rowcolor[gray]{0.95}
		60 & 4 & 8 & 4 & 12 & 4 & 8 & 0 & 16 \\ 
		63 & 5 & 10 & 8 & 10 & 0 & 16 & 5 & 10 \\ 
		\rowcolor[gray]{0.95}
		64 & 7/2 & 13/2 & 0 & 9 & 7 & 0 & 0 & 10 \\ 
		67 & 1 & 0 & 0 & 0 & 0 & 0 & 0 & 0 \\ 
		\rowcolor[gray]{0.95}
		68 & 4 & 4 & 8 & 0 & 0 & 8 & 8 & 0 \\ 
		71 & 7 & 14 & 14 & 14 & 14 & 28 & 0 & 14 \\ 
		\rowcolor[gray]{0.95}
		72 & 3 & 3 & 6 & 0 & 0 & 6 & 0 & 0 \\ 
		75 & 7/3 & 0 & 7/3 & 0 & 4 & 0 & 14/3 & 0 \\ 
		\rowcolor[gray]{0.95}
		76 & 4 & 6 & 0 & 6 & 8 & 0 & 8 & 0 \\ 
		79 & 5 & 10 & 0 & 10 & 10 & 0 & 0 & 10 \\ 
		\rowcolor[gray]{0.95}
		80 & 6 & 10 & 12 & 12 & 6 & 20 & 12 & 8 \\ 
		83 & 3 & 0 & 6 & 0 & 0 & 0 & 6 & 0 \\ 
		\rowcolor[gray]{0.95}
		84 & 4 & 4 & 4 & 0 & 8 & 4 & 4 & 0 \\ 
		87 & 6 & 12 & 6 & 0 & 0 & 12 & 12 & 12 \\ 
		\rowcolor[gray]{0.95}
		88 & 2 & 2 & 0 & 0 & 0 & 0 & 0 & 0 \\ 
		91 & 2 & 0 & 0 & 0 & 4 & 0 & 2 & 0 \\ 
		\rowcolor[gray]{0.95}
		92 & 6 & 12 & 12 & 18 & 0 & 24 & 0 & 24 \\ 
		95 & 8 & 16 & 16 & 16 & 8 & 32 & 0 & 16 \\ 
		\rowcolor[gray]{0.95}
		96 & 6 & 10 & 6 & 12 & 12 & 10 & 12 & 8 \\ 
		99 & 3 & 0 & 6 & 0 & 6 & 0 & 0 & 0 \\ 
		\rowcolor[gray]{0.95}
		100 & 5/2 & 5/2 & 0 & 0 & 5 & 0 & 0 & 0 \\ 
		\hline\noalign{\smallskip}
	\end{tabular}
\end{table}

\begin{table}[htb]
	\caption{$ H^M(D) $ for positive integers $ D \le 100 $.}
	\label{tab:Hurwitz_class_num2}
	\centering
	\begin{tabular}{cccccccc}
		\hline\noalign{\smallskip}
		$ D $ & $ H^{9}(D) $ & $ H^{10}(D) $ & $ H^{12}(D) $ & $ H^{13}(D) $ & $ H^{16}(D) $ & $ H^{18}(D) $ & $ H^{25}(D) $ \\
		\hline
		\rowcolor[gray]{0.95}
		0 & $ -1 $ & $ -3/2 $ & $ -2 $ & $ -7/6 $ & $ -2 $ & $ -3 $ & $ -5/2 $ \\ 
		3 & 0 & 0 & 0 & 2/3 & 0 & 0 & 0 \\ 
		\rowcolor[gray]{0.95}
		4 & 0 & 1 & 0 & 1 & 0 & 0 & 1 \\ 
		7 & 0 & 0 & 0 & 0 & 2 & 0 & 0 \\ 
		\rowcolor[gray]{0.95}
		8 & 2 & 0 & 0 & 0 & 0 & 2 & 0 \\ 
		11 & 2 & 0 & 0 & 0 & 0 & 0 & 2 \\ 
		\rowcolor[gray]{0.95}
		12 & 0 & 0 & 2 & 8/3 & 0 & 0 & 0 \\ 
		15 & 0 & 4 & 4 & 0 & 4 & 0 & 0 \\ 
		\rowcolor[gray]{0.95}
		16 & 0 & 5 & 0 & 3 & 0 & 0 & 3 \\ 
		19 & 0 & 0 & 0 & 0 & 0 & 0 & 2 \\ 
		\rowcolor[gray]{0.95}
		20 & 4 & 2 & 0 & 0 & 0 & 4 & 0 \\ 
		23 & 6 & 0 & 12 & 6 & 6 & 12 & 0 \\ 
		\rowcolor[gray]{0.95}
		24 & 0 & 4 & 0 & 0 & 0 & 0 & 4 \\ 
		27 & 4 & 0 & 0 & 8/3 & 0 & 0 & 0 \\ 
		\rowcolor[gray]{0.95}
		28 & 0 & 0 & 0 & 0 & 8 & 0 & 0 \\ 
		31 & 0 & 12 & 0 & 0 & 6 & 0 & 6 \\ 
		\rowcolor[gray]{0.95}
		32 & 6 & 0 & 12 & 0 & 0 & 10 & 0 \\ 
		35 & 4 & 0 & 0 & 4 & 0 & 0 & 0 \\ 
		\rowcolor[gray]{0.95}
		36 & 6 & 5 & 0 & 5 & 0 & 6 & 5 \\ 
		39 & 0 & 16 & 8 & 4 & 8 & 0 & 8 \\ 
		\rowcolor[gray]{0.95}
		40 & 0 & 2 & 0 & 4 & 0 & 0 & 0 \\ 
		43 & 0 & 0 & 0 & 2 & 0 & 0 & 0 \\ 
		\rowcolor[gray]{0.95}
		44 & 8 & 12 & 12 & 0 & 0 & 12 & 8 \\ 
		47 & 10 & 0 & 20 & 0 & 10 & 20 & 0 \\ 
		\rowcolor[gray]{0.95}
		48 & 0 & 0 & 7 & 20/3 & 8 & 0 & 0 \\ 
		51 & 0 & 0 & 0 & 4 & 0 & 0 & 4 \\ 
		\rowcolor[gray]{0.95}
		52 & 0 & 0 & 0 & 2 & 0 & 0 & 0 \\ 
		55 & 0 & 8 & 0 & 8 & 8 & 0 & 0 \\ 
		\rowcolor[gray]{0.95}
		56 & 8 & 8 & 0 & 8 & 0 & 8 & 8 \\ 
		59 & 6 & 0 & 0 & 0 & 0 & 0 & 6 \\ 
		\rowcolor[gray]{0.95}
		60 & 0 & 8 & 12 & 0 & 16 & 0 & 0 \\ 
		63 & 12 & 0 & 16 & 0 & 10 & 24 & 0 \\ 
		\rowcolor[gray]{0.95}
		64 & 0 & 13 & 0 & 7 & 12 & 0 & 7 \\ 
		67 & 0 & 0 & 0 & 0 & 0 & 0 & 0 \\ 
		\rowcolor[gray]{0.95}
		68 & 8 & 0 & 0 & 8 & 0 & 8 & 0 \\ 
		71 & 14 & 28 & 28 & 0 & 14 & 28 & 14 \\ 
		\rowcolor[gray]{0.95}
		72 & 12 & 0 & 0 & 0 & 0 & 12 & 0 \\ 
		75 & 0 & 0 & 0 & 14/3 & 0 & 0 & 10 \\ 
		\rowcolor[gray]{0.95}
		76 & 0 & 12 & 0 & 0 & 0 & 0 & 8 \\ 
		79 & 0 & 20 & 0 & 10 & 0 & 0 & 10 \\ 
		\rowcolor[gray]{0.95}
		80 & 12 & 10 & 24 & 0 & 24 & 20 & 0 \\ 
		83 & 6 & 0 & 0 & 0 & 0 & 0 & 0 \\ 
		\rowcolor[gray]{0.95}
		84 & 0 & 8 & 0 & 0 & 0 & 0 & 8 \\ 
		87 & 0 & 0 & 12 & 12 & 12 & 0 & 0 \\ 
		\rowcolor[gray]{0.95}
		88 & 0 & 0 & 0 & 4 & 0 & 0 & 0 \\ 
		91 & 0 & 0 & 0 & 2 & 0 & 0 & 4 \\ 
		\rowcolor[gray]{0.95}
		92 & 12 & 0 & 36 & 12 & 36 & 24 & 0 \\ 
		95 & 16 & 16 & 32 & 16 & 32 & 32 & 0 \\ 
		\rowcolor[gray]{0.95}
		96 & 0 & 20 & 10 & 0 & 12 & 0 & 12 \\ 
		99 & 12 & 0 & 0 & 0 & 0 & 0 & 6 \\ 
		\rowcolor[gray]{0.95}
		100 & 0 & 5 & 0 & 5 & 0 & 0 & 15 \\ 
		\hline\noalign{\smallskip}
	\end{tabular}				
\end{table}

% --------------------------------------------------------------------------

\section{Examples} \label{sec:exapmles}

% --------------------------------------------------------------------------

In this section, we give several examples of our formula in Theorem \ref{thm:main_class_num} and give conjectures for a square $ N $. 
% in the case when the level is $ p = 2, 3, 5, 7 $ or $ 13 $. 

To extend our formula in Theorem \ref{thm:main_class_num} for a square $ N $, we need to define $ H^M(0) $. 

In the case when the level is 1, put the 0th Hurwitz class number $ H(0) := -1/12 $.
Then Hurwitz-Eichler relation (\ref{eq:Hurwitz-Eichler}) holds for a square $ N $:
\[
\sum_{x \in \Z,\ x^2 \le 4N} H(4N- x^2)
= \sum_{ad=N} \max \{ a, d \}.
\]

Similarly, we define the Hurwitz class number $ H^M(0) $ for
$ M $ with $ 2 \le M \le 10 $ or $ M \in \{ 12, 13, 16, 18, 25 \} $ by
\begin{equation} \label{eq:H^M(0)}
	H^M(0) := -\frac{[\SL_2(\Z) : \Gamma_0(M)]}{12}
	= -\frac{M}{12} \prod_{p \mid M} \left( 1 + \frac{1}{p} \right).
\end{equation}

Under this definition, we calculate
\[
S(N) := \sum_{x \in \Z,\ x^2 \le 4N} H \left( 4N - x^2 \right), \quad  
S^{M}(N) := \sum_{x \in \Z,\ x^2 \le 4N} H^{M} \left( 4N - x^2 \right)
\]
for a positive integer $ N \le 12 $ in Table \ref{tab:Hurwitz_class_num_sum1} and Table \ref{tab:Hurwitz_class_num_sum2}. 
We can confirm that Theorem \ref{thm:main_class_num} holds for square-free $ N $ coprime to $ M $.

\begin{table}[htb]
	\caption{$ S(N) $ and $ S^{M}(N) $ for positive integers $ N \le 25 $.}
	\label{tab:Hurwitz_class_num_sum1}
	\centering
	\begin{tabular}{ccccccccc}
		\hline\noalign{\smallskip}
		$ N $ & $ S(N) $ & $ S^{2}(N) $ & $ S^{3}(N) $ & $ S^{4}(N) $ & $ S^{5}(N) $ & $ S^{6}(N) $ & $ S^{7}(N) $ & $ S^{8}(N) $ \\
		\hline
		\rowcolor[gray]{0.95}
		1 & 1 & 0 & 0 & $ -1 $ & 0 & $ -2 $ & 0 & $ -2 $ \\ 
		2 & 4 & 6 & 2 & 4 & 2 & 2 & 2 & 4 \\ 
		\rowcolor[gray]{0.95}
		3 & 6 & 4 & 10 & 2 & 4 & 6 & 4 & 0 \\ 
		4 & 10 & 18 & 6 & 18 & 6 & 10 & 6 & 12 \\ 
		\rowcolor[gray]{0.95}
		5 & 10 & 8 & 8 & 6 & 18 & 4 & 8 & 4 \\ 
		6 & 18 & 28 & 30 & 20 & 12 & 46 & 12 & 20 \\ 
		\rowcolor[gray]{0.95}
		7 & 14 & 12 & 12 & 10 & 12 & 8 & 26 & 8 \\ 
		8 & 24 & 46 & 18 & 52 & 18 & 34 & 18 & 52 \\ 
		\rowcolor[gray]{0.95}
		9 & 21 & 16 & 40 & 11 & 16 & 30 & 16 & 6 \\ 
		10 & 30 & 48 & 24 & 36 & 54 & 36 & 24 & 36 \\ 
		\rowcolor[gray]{0.95}
		11 & 22 & 20 & 20 & 18 & 20 & 16 & 20 & 16 \\ 
		12 & 44 & 80 & 74 & 83 & 32 & 134 & 32 & 64 \\ 
		\rowcolor[gray]{0.95}
		13 & 26 & 24 & 24 & 20 & 24 & 20 & 24 & 20 \\ 
		14 & 42 & 68 & 36 & 52 & 36 & 56 & 78 & 52 \\ 
		\rowcolor[gray]{0.95}
		15 & 40 & 32 & 68 & 24 & 72 & 52 & 32 & 16 \\ 
		16 & 52 & 102 & 42 & 118 & 42 & 82 & 42 & 132 \\ 
		\rowcolor[gray]{0.95}
		17 & 34 & 32 & 32 & 30 & 32 & 28 & 32 & 28 \\ 
		18 & 66 & 106 & 126 & 80 & 54 & 202 & 54 & 80 \\ 
		\rowcolor[gray]{0.95}
		19 & 38 & 36 & 36 & 34 & 36 & 32 & 36 & 32 \\ 
		20 & 70 & 128 & 56 & 136 & 126 & 100 & 56 & 108 \\ 
		\rowcolor[gray]{0.95}
		21 & 56 & 48 & 96 & 38 & 48 & 80 & 104 & 32 \\ 
		22 & 66 & 108 & 60 & 60 & 60 & 96 & 60 & 84 \\ 
		\rowcolor[gray]{0.95}
		23 & 46 & 44 & 44 & 42 & 44 & 40 & 44 & 40 \\ 
		24 & 100 & 192 & 170 & 196 & 80 & 326 & 80 & 224 \\ 
		\rowcolor[gray]{0.95}
		25 & 55 & 48 & 48 & 41 & 108 & 34 & 48 & 34 \\ 
		\hline\noalign{\smallskip}
	\end{tabular}		
\end{table}

\begin{table}[htb]
	\caption{$ S^{M}(N) $ for positive integers $ N \le 25 $.}
	\label{tab:Hurwitz_class_num_sum2}
	\centering
	\begin{tabular}{cccccccc}
		\hline\noalign{\smallskip}
		$ N $ & $ S^{9}(N) $ & $ S^{10}(N) $ & $ S^{12}(N) $ & $ S^{13}(N) $ & $ S^{16}(N) $ & $ S^{18}(N) $ & $ S^{25}(N) $ \\
		\hline
		\rowcolor[gray]{0.95}
		1 & $ -2 $ & $ -2 $ & $ -4 $ & 0 & $ -4 $ & $ -6 $ & $ -4 $ \\ 
		2 & 2 & 2 & 0 & 2 & 4 & 2 & 2 \\ 
		\rowcolor[gray]{0.95}
		3 & 8 & 0 & 2 & 4 & 0 & 4 & 4 \\ 
		4 & $ -2 $ & 10 & 8 & 6 & 8 & $ -6 $ & $ -2 $ \\ 
		\rowcolor[gray]{0.95}
		5 & 8 & 14 & 0 & 8 & 0 & 4 & 16 \\ 
		6 & 24 & 16 & 32 & 12 & 20 & 36 & 4 \\ 
		\rowcolor[gray]{0.95}
		7 & 8 & 8 & 4 & 12 & 8 & 0 & 12 \\ 
		8 & 18 & 34 & 36 & 18 & 44 & 34 & 18 \\ 
		\rowcolor[gray]{0.95}
		9 & 44 & 6 & 20 & 16 & $ -4 $ & 28 & 4 \\ 
		10 & 12 & 86 & 24 & 24 & 36 & 12 & 48 \\ 
		\rowcolor[gray]{0.95}
		11 & 20 & 16 & 12 & 20 & 16 & 16 & 12 \\ 
		12 & 60 & 56 & 139 & 32 & 56 & 108 & 32 \\ 
		\rowcolor[gray]{0.95}
		13 & 20 & 20 & 14 & 50 & 16 & 12 & 24 \\ 
		14 & 36 & 56 & 40 & 36 & 52 & 56 & 20 \\ 
		\rowcolor[gray]{0.95}
		15 & 56 & 56 & 36 & 32 & 16 & 40 & 64 \\ 
		16 & 22 & 82 & 90 & 42 & 132 & 42 & 18 \\ 
		\rowcolor[gray]{0.95}
		17 & 32 & 28 & 24 & 32 & 24 & 28 & 32 \\ 
		18 & 144 & 82 & 152 & 54 & 80 & 228 & 54 \\ 
		\rowcolor[gray]{0.95}
		19 & 32 & 32 & 28 & 36 & 32 & 24 & 36 \\ 
		20 & 56 & 230 & 104 & 56 & 104 & 100 & 112 \\ 
		\rowcolor[gray]{0.95}
		21 & 80 & 32 & 62 & 48 & 64 & 64 & 40 \\ 
		22 & 48 & 96 & 72 & 60 & 64 & 72 & 60 \\ 
		\rowcolor[gray]{0.95}
		23 & 44 & 40 & 36 & 44 & 52 & 40 & 44 \\ 
		24 & 140 & 152 & 370 & 80 & 308 & 268 & 40 \\ 
		\rowcolor[gray]{0.95}
		25 & 34 & 94 & 16 & 48 & 44 & 6 & 126 \\
		\hline\noalign{\smallskip}
	\end{tabular}		
\end{table}

Here we have the following conjecture which treats the case when $ N $ is a square. 

\begin{conjecture}
	Let $ M $ be $ 2 \le M \le 10 $ or $ M \in \{ 12, 13, 16, 18, 25 \} $ and  $ N $ be a positive integer coprime to $ M $.
%	Let $ M $ be $ 2 \le M \le 10 $ or $ M = 13 $.
	We put the number $ \delta_M(1, N) $ as in Theorem $ \ref{thm:main_cusp} $ and the Hurwitz class number $ H^M(0) $ as in  $ (\ref{eq:H^M(0)}) $.
	Unless $ M = 25 $ and $ N \equiv \pm 1 \bmod 5 $, we have
	\[
	\sum_{x \in \Z,\ x^2 \le 4N} H^M \left( 4N - x^2 \right)
	= \sum_{a d = 1} \left(\max \{ a, d \} - \delta_M(1, N) \min \{ a, d \} \right)
	\]
	and if $ M = 25 $ and $ N \equiv \pm 1 \bmod 5 $, we have	
	\[
	\sum_{x \in \Z,\ x^2 \le 4N} 
	H^{25} \left( 4N - x^2 \right)
	= \sum_{a d = N} |a - d| 
	- 4\sum_{ad = N, a \equiv d \bmod 5}
	\min \left\{ a, d \right\}.
	\]
\end{conjecture}

% --------------------------------------------------------------------------$

\bibliographystyle{plain}
\bibliography{myrefs_for_class_number}

% --------------------------------------------------------------------------
\end{document}